\documentclass{amsart}
\usepackage{amssymb, amsmath, latexsym, mathabx, dsfont, tikz}
\usepackage{multirow}
\usepackage{setspace}
\usepackage{enumerate}
\usepackage{diagbox}
\usepackage{stmaryrd}
\usepackage{caption}
\usetikzlibrary{arrows.meta}

\renewcommand{\baselinestretch}{\baselinestretch}
\renewcommand{\baselinestretch}{1.1}
\numberwithin{equation}{section}

\newtheorem{thm}{Theorem}[section]
\newtheorem{lem}[thm]{Lemma}

\newtheorem{prop}[thm]{Proposition}

\theoremstyle{remark}
\newtheorem{rmk}[thm]{Remark}

\numberwithin{equation}{section}

\newcommand{\ra}{{\ \longrightarrow \ }}
\newcommand{\nra}{{\ \longarrownot\longrightarrow \ }}
\newcommand{\ratwo}{{\ \overset{2}{\longrightarrow} \ }}
\newcommand{\nratwo}{{\ \overset{2}{\longarrownot\longrightarrow} \ }}
\newcommand{\ral}[1]{\overset{\lambda_{#1}}{\ra}}
\newcommand{\lral}[1]{\overset{\lambda_{#1}}{\longleftrightarrow}}
\newcommand{\ord}{\text{ord}}
\newcommand{\sort}{\text{Sort}}
\newcommand{\n}{{\mathbb N}}
\newcommand{\z}{{\mathbb Z}}
\newcommand{\q}{{\mathbb Q}}
\newcommand{\E}{{\mathcal E}}
\newcommand{\T}{{\mathcal T}}
\newcommand{\Mod}[1]{\ \left(\mathrm{mod}\ #1 \right)}
\newcommand{\ltra}{{\ \relbar\joinrel\twoheadrightarrow \ }}

\begin{document}
\title{Regular triangular forms of rank exceeding 3}

\author{Mingyu Kim}

\address{Department of Mathematics, Sungkyunkwan University, Suwon 16419, Korea}
\email{kmg2562@skku.edu}

\thanks{This work was supported by the National Research Foundation of Korea(NRF) grant funded by the Korea government(MSIT) (NRF-2021R1C1C2010133).}



\begin{abstract} 
A triangular form is defined to be an integer-valued quadratic polynomial of the form $a_1P_3(x_1)+a_2P_3(x_2)+\cdots+a_kP_3(x_k)$ where $a_i's$ are positive integers and $P_3(x)=x(x+1)/2$.
A triangular form is called regular if it represents all positive integers which are locally represented.
In this article, we find all regular triangular forms of more than three variables.
\end{abstract}

\maketitle

\section{Introduction} \label{secint}
For a polynomial $P_3(x)=x(x+1)/2$ and positive integers $a_1,a_2,\dots,a_k$, we call the integer-valued quadratic polynomial
$$
f=f(x_1,x_2,\dots,x_k)=a_1P_3(x_1)+a_2P_3(x_2)+\cdots+a_kP_3(x_k)
$$
{\it a $k$-ary triangular form}, following \cite{CO13}.
For a nonnegative integer $n$ and a $k$-ary triangular form $f$, we say that {\it $n$ is represented by $f$} if the Diophantine equation
\begin{equation} \label{eq1}
f(x_1,x_2,\dots,x_k)=n
\end{equation}
has an integer solution.
A triangular form that represents all positive integers is called {\it universal}.
Throughout, the triangular form $a_1P_3(x_1)+\cdots+a_kP_3(x_k)$ will be denoted by $\mathcal T(a_1,a_2,\dots,a_k)$.

In 1796, Gauss proved the so-called Eureka Theorem which states that every positive integer is the sum of three triangular numbers.
Under our notation, it is equivalent to saying that the ternary triangular form $\mathcal T(1,1,1)$ is universal.
This theorem was generalized by Liouville; he shows that there are exactly seven universal ternary triangular forms, which are
$$
\mathcal T(1,1,1),\ \mathcal T(1,1,2),\ \mathcal T(1,1,4),\ \mathcal T(1,1,5),\ \mathcal T(1,2,2),\ \mathcal T(1,2,3),\ \mathcal T(1,2,4).
$$
Bosma and Kane \cite{BK} classified all universal triangular forms by  proving that all of the four quaternary triangular forms
$$
\mathcal T(1,1,3,3),\ \mathcal T(1,1,3,6),\ \mathcal T(1,1,3,7),\ \mathcal T(1,1,3,8)
$$
are universal.
Note that any universal triangular form can be obtained by adding some ``redundant" triangular forms to one of the above $7+4=11$ forms.

One may also consider the solvability of Equation \eqref{eq1} over the ring of $p$-adic integers $\z_p$ for a prime $p$.
We say that {\it $n$ is locally represented by $f$} if \eqref{eq1} is solvable over $\z_p$ for every prime $p$.
A triangular form $f=\mathcal T(a_1,a_2,\dots,a_k)$ is called {\it primitive} if its coefficients are coprime, i.e., $\gcd(a_1,a_2,\dots,a_k)=1$.
Note that the primitive ternary triangular form $\mathcal T(1,1,3)$ cannot represent (over $\z$) any integer $n$ in the set
$$
A=\left\{ \frac{3^{2a+1}(3b+2)-5}{8} : a,b\in \n \cup \{0\} \right\},
$$
since $f$ does not represent any integer in $A$ over $\z_3$.
This can be shown by using the equivalence
$$
\frac{x(x+1)}{2}+\frac{y(y+1)}{2}+3\frac{z(z+1)}{2}=n \Leftrightarrow (2x+1)^2+(2y+1)^2+3(2z+1)^2=8n+5
$$
and the fact that diagonal ternary quadratic form $x^2+y^2+3z^2$ does not represent integers of the form $3^{2a+1}(3b+2)$ ($a,b\in \z_{\ge 0}$) over $\z_3$.
However, $\mathcal T(1,1,3)$ do represent (over $\z$) all positive integers not in the above set $A$ (for this, see \cite[Theorem 3.17]{KO1}).
A triangular form $f$ is said to be {\it regular} if it represents every positive integer which is locally represented.
Note that for a triangular form $f=\mathcal T(a_1,a_2,\dots,a_k)$ and a positive integer $r$, the form $rf=\mathcal T(ra_1,ra_2,\dots,ra_k)$ is regular if and only if $f$ is.
This shows that when we find regular triangular forms, we may only consider primitive ones.
The finiteness of primitive regular ternary triangular forms was established by Chan and Oh \cite{CO13}, and the complete enumeration was done by Oh and the author \cite{KO1}.

\begin{thm}[Theorem 4.10 of \cite{KO1}] \label{thmrk3}
There are exactly 49 primitive regular ternary triangular forms, as listed in Table \ref{tablerk3}.
\end{thm}

We now turn to representations by quadratic forms.
A positive definite integral quadratic form
$$
g=g(x_1,x_2,\dots,x_k)=\sum_{1\le i,j\le k} a_{ij}x_ix_j\ (a_{ij}=a_{ji}\in \z)
$$
is said to represent a nonnegative integer $n$ if the Diophantine equation
$$
g(x_1,x_2,\dots,x_k)=n
$$
has an integer solution.
Throughout, a quadratic form always refers to a positive definite integral quadratic form, unless otherwise stated.
The notion of universality and the regularity of a quadratic form is as for the case of triangular forms.
Lagrange's four-square theorem says that the quadratic form $x^2+y^2+z^2+w^2$ is universal.
Ramanujan \cite{R} found all 54 universal diagonal quadratic forms (see also \cite{D}).
Conway and Schneeberger announced the so-called 15-theorem which says that a quadratic form is universal if it represents all positive integers up to 15.
Bhargava \cite{B} found a simpler proof for the 15-theorem.
Meanwhile, he showed that any universal quadratic form $f$ contains a universal quadratic form $g$ of rank less than or equal to 5 as its subform.

Note that a universal quadratic form is nothing but a regular quadratic form which is locally universal.
For a regular quadratic form $g$, we call $g$ {\it old} if there is a proper subform $h$ of $g$ such that $h$ represents all integers that are represented by $g$, and $g$ is called {\it new} if it is not old.
Under these notations, we can say that the rank of new regular quadratic forms that are locally universal is less than or equal to 5.

A classification of quaternary regular diagonal quadratic forms was carried out by B.M. Kim \cite{BMK}, but the paper was never published.
Recently, Oh and the author \cite{KO2} proved that there is a constant $C$ (which is implicit) such that the rank of any new regular quadratic form is less than $C$.
The author, in \cite{K1}, classified all regular ``diagonal" quadratic forms of rank $\ge 4$.
Note that the rank of new regular ``diagonal" quadratic forms is less than or equal to 5.

We now get back to the world of triangular forms.
We first define the notion of newness for triangular forms analogously to the case of quadratic forms.
A regular $k$-ary triangular form $f=\mathcal T(a_1,a_2,\dots,a_k)$ is called {\it old} if there is an index $i$ with $1\le i \le k$ such that the set of positive integers represented by $f$ is equal to that of the $(k-1)$-ary triangular form $g=\mathcal T(a_1,\dots,a_{i-1},\widehat{a_i},a_{i+1},\dots,a_k)$, where $\widehat{a_i}$ means that the component $a_i$ is missing.
We call $f$ {\it new} if it is not old.
As mentioned before, we already have the list of regular ternary triangular forms.
Note that any ternary regular triangular form is new (see Proposition \ref{propter}).
So now our tasks are the following;
\begin{enumerate} [(i)]
\item find all new regular triangular forms of rank exceeding 3;
\item describe all regular triangular forms using the result of (i).
\end{enumerate}
The mission will be completed by the three consecutive theorems below.

\begin{thm} \label{thmrk4}
A primitive quaternary triangular form $\T(a_1,a_2,a_3,a_4)$ with $a_1\le a_2\le a_3\le a_4$ is new regular if and only if it appears in Table \ref{tablerk4}.
\end{thm}

To find all regular quaternary triangular forms, we will engage a transformation which preserves the regularity of triangular forms (see Subsection \ref{subsecwatson}).
Note that the transformation was used by Watson \cite{W1} when he showed the finiteness of regular ternary quadratic forms (up to equivalence).
And then we define Watson transformation on triangular forms.
We will see that if a regular triangular forms of rank greater than 3 is given, then by taking Watson transformations finitely many times, one may obtain a universal triangular forms of the same rank.
So one may start at universal triangular forms of some rank and trace back the transformations to get all regular triangular forms of that rank.
However, there is a ``small" problem, that is, there are infinitely many universal triangular forms, e.g., the form $\mathcal T(1,1,1,n)$ is universal (though it is old) for every positive integer $n$.
There is another problem if we consider the regularity and the newness of regular triangular forms simultaneously.
Watson transformations for sure preserve the regularity, but it does not ensure the newness.
In other words, it is possible that a new regular triangular form goes to an old regular triangular form under the Watson transformation.
To handle the issues, we define drop structures in Subsection \ref{subsecdrop} which indicate when a new regular form is transformed to an old one.

\begin{thm} \label{thmrk5}
For any $k\ge 5$, there is no new regular $k$-ary triangular form.
\end{thm}

To prove this statement, we will follow (after suitable modification) the way to discover all new regular diagonal quadratic forms starting from each candidate for ternary section of regular diagonal quadratic forms of rank $\ge 4$, described in \cite{K1}.
As mentioned, we first find all possible candidates for a ternary section $\mathcal T(a_1,a_2,a_3)$ of any (new) regular triangular form $f=\mathcal T(a_1,a_2,\dots,a_k)$ of rank $k\ge 4$ with $a_1\le a_2\le \cdots \le a_k$.
After that, we need to look at local representations of ternary sections to determine the possibilities for the fourth coefficient $a_4$.
And then the existence of the fifth coefficient $a_5$ will lead us to ``redundancy of coefficients" which contradicts to the newness of $f$.

For the definition of $\xi$ in the following theorem, see Remark \ref{rmkdivisor}.

\begin{thm} \label{thmold}
Let $\E=\T(a_1,a_2,\dots,a_k)$ be a triangular form with $a_1\le \cdots \le a_k$.
Then $\E$ is regular if and only if there is a subset $I=\{i_1,i_2,\dots,i_{\vert I\vert}\}$ of the set $\{1,2,\dots,k\}$ with $\vert I\vert \in \{3,4\}$ such that $(a_{i_1},a_{i_2},\dots,a_{i_{\vert I\vert}})$ appears in either Table \ref{tablerk3} or Table \ref{tablerk4}, and $a_j\equiv 0\Mod {\xi(a_{i_1},a_{i_2},\dots,a_{i_{\vert I\vert}})}$ for every $j\in \{1,2,\dots,k\}-I$.
\end{thm}

Note that the quantity $\xi(b_1,b_2,\dots,b_r)$ for any vector $(b_1,b_2,\dots,b_r)\in \n^r$ in this theorem can be effectively computed using \cite[Theorem 1]{OM2}.
Hence Theorem \ref{thmold} gives us an effective criterion determining whether the given triangular form of rank $\ge 3$ is regular or not, provided that we have Table \ref{tablerk3} and Table \ref{tablerk4} at hand.

In Section \ref{secpre}, we introduce notations and terminologies used throughout the article. Some lemmas on representations by quadratic forms will also be addressed.
Section \ref{secriv} is devoted to develop the concept of rivers of new regular triangular forms.
To this end, we need to look at the criterion for newness (or equivalently, oldness) of regular triangular forms (see Lemma \ref{lemnew}), Watson transformations, triangular forms which play a role of river mouth, and the notion of drop structures which are built on that rivers.
In Section \ref{secqua}, we prove Theorem \ref{thmrk4}.
Section \ref{sechigh} contains the proofs for Theorems \ref{thmrk5} and \ref{thmold}.

\section{Preliminaries} \label{secpre}

We adopt the geometric language of quadratic spaces and lattices.
Let $R$ be the ring of rational integers $\z$ or the ring of $p$-adic integers $\z_p$ for a prime $p$, and let $F$ be the field of fractions of $R$.
A quadratic ($F$-)space is a finite dimensional vector space $V$ over $F$ equipped with a quadratic map $Q:V\times V\to F$.
We usually omit the quadratic map $Q$ and just write $(V,Q)$ simply as $V$, when no confusion will arise.
Associated symmetric bilinear map $B:V\times V\to F$ is defined by
$$
B(\mathbf{x},\mathbf{y})=\frac{1}{2}(Q(\mathbf{x}+\mathbf{y})-Q(\mathbf{x})-Q(\mathbf{y})).
$$
An $R$-lattice $L$ is a finitely generated free $R$-submodule of $V$.
The scale of $L$ is an $R$-module generated by the subset $B(L,L)=\{B(\mathbf{x},\mathbf{y})\in F : \mathbf{x},\mathbf{y}\in L\}$ of $F$ and is denoted by $\mathfrak{s}L$.
We call $L$ {\it integral} if $\mathfrak{s}L\subseteq R$, and {\it primitive} if $\mathfrak{s}L=R$.
Throughout, we assume that any $R$-lattice is integral.
Let $\{ \mathbf{x}_i\}_{i=1}^{k}$ be an $R$-basis of $L$.
The matrix $(B(\mathbf{x}_i,\mathbf{x}_j))_{1\le i,j\le k}$ is called the {\it Gram matrix of $L$} with respect to this basis.
A diagonal matrix with diagonal entries $a_1,a_2,\dots,a_k$ will be denoted by $\langle a_1,a_2,\dots,a_k\rangle$.
We abuse the notation and write $L=A$, where $L$ is an $R$-lattice and $A$ is the Gram matrix of $L$ in some basis of $L$.

For $\gamma \in F$ and a quadratic space $V$, we say that {\it $\gamma$ is represented by $V$} if there is a vector $\mathbf{v}\in V$ such that $Q(\mathbf{v})=\gamma$.
A quadratic space $V$ is called {\it universal} if every element in $F$ is represented by $V$.
Let $L$ be an $R$-lattice.
For $\gamma \in R$, we say that {\it $\gamma$ is represented by $R$} if $\gamma=Q(\mathbf{v})$ for some vector $\mathbf{v}$ in $R$.
We say {\it $L$ is universal} if $Q(L)=R$, where $Q(L)=\{ Q(\mathbf{v}): \mathbf{v}\in R\}$.

Now let $L=\z \mathbf{x}_1+\z \mathbf{x}_2+\cdots+\z \mathbf{x}_k$ be a $\z$-lattice and $p$ be a prime.
The $\z_p$-lattice $L_p$ and quadratic $\q_p$-space $\q_pL$ is given by
\begin{align*}
L_p&=L\otimes \z_p=\z_p \mathbf{x}_1+\z_p \mathbf{x}_2+\cdots+\z_p \mathbf{x}_k,\\
\q_pL&=L\otimes \q_p=\q_p \mathbf{x}_1+\q_p \mathbf{x}_2+\cdots+\q_p \mathbf{x}_k.
\end{align*}
Any $n\in \z$ is said to be {\it represented by $L$ over $\z_p$} if $n\in Q(L_p)$.
For an integer $u$, we say that {\it $L$ is $u\z_p$-universal} if $L_p$ represents every $\gamma \in  u\z_p$.
For $p\neq 2$, we let $\Delta_p$ be a fixed non-square unit in $\z_p^{\times}$.
Any unexplained notations and terminologies on quadratic spaces and lattices can be found in \cite{Ki} or \cite{OM}.

\begin{lem} \label{lemodd1}
Let $p$ be an odd prime and $K$ be a $\z_p$-lattice.
For $\gamma \in \z_p$, $Q(K)=Q(K\perp \langle \gamma \rangle)$ if and only if $\gamma \z_p \subseteq Q(K)$.
\end{lem}

\begin{proof}
See \cite[Lemma 2.1]{K1}.
\end{proof}


\begin{lem} \label{lemodd2}
Let $p$ be an odd prime.
Let $K$ be a $\z_p$-lattice and $\gamma \in \z_p$ such that $\gamma \z_p \not\subseteq Q(K)$. Then there exists an element $\beta \in \z_p$ such that $\beta \not\in Q(K)$ and $\beta +\gamma \in Q(K\perp \langle \gamma \rangle)$.
\end{lem}

\begin{proof}
First, assume that $p\gamma \z_p\not\subseteq Q(K)$.
Then there exists an element $\eta \in \z_p$ with $\ord_p\eta >\ord_p\gamma$ such that $\eta \nra Q(K)$.
By Local Square Theorem, $\eta+\gamma=\gamma \delta^2$ for some $\delta \in \z_p^{\times}$.
So we are done by taking $\beta=\eta$. 
From now on, we always assume that $p\gamma \z_p \subset Q(K)$.

Second, assume that $\gamma \nra Q(K)$ and $\gamma \Delta_p\nra Q(K)$.
We take $\beta=8\gamma$. 
Since $\beta \in \gamma \z_p^{\times}$, we have $\beta \nra K$.
On the other hand,
$$
\beta+\gamma=9\gamma\in Q(\langle \gamma \rangle)\subseteq Q(K\perp \langle \gamma \rangle).
$$

Last, assume that there is an $\epsilon \in \z_p^{\times}$ such that $\gamma \epsilon \ra K$ and $\gamma \epsilon \Delta_p \nra K$.
We take $\beta=\gamma \epsilon \Delta_p$.
Then we have
$$
\beta+\gamma \in \gamma \z_p\subseteq Q(K\perp \langle \gamma \rangle),
$$
where the last inclusion follows from
$$
\gamma \z_p^{\times}\subset Q(\langle \gamma \epsilon \rangle \perp \langle \gamma \rangle)\subset Q(K\perp \langle \gamma \rangle)
$$
and our assumption that $p\gamma \z_p\subset Q(K)$.
This completes the proof.
\end{proof}


For a $\z$-lattice $K$ and a positive integer $m$,
$$
\Lambda_m(K)=\{ \mathbf{x}\in K : Q(\mathbf{x}+\mathbf{z})\equiv Q(\mathbf{z})\Mod m\ \text{for all}\ \mathbf{z}\in K\}
$$
is a $\z$-sublattice of $K$ which is called the Watson transformation of $K$ modulo $m$.
The primitive $\z$-lattice obtained from $\Lambda_m(K)$ by scaling the quadratic space $\q \otimes K$ by a suitable rational number is denoted by $\lambda_m(K)$.


\begin{lem} \label{lemqzq}
Let $p$ be an odd prime and $K$ be a $\z$-lattice.
Then we have the following:
\begin{enumerate} [(i)]
\item For any odd prime $q\neq p$, $\Lambda_p(K)_q=K_q$.
\item If $Q(K_p)\subsetneq \z_p$, then $p\z_p\cap Q(K_p)=Q(\Lambda_p(K)_p)$.
\end{enumerate}
\end{lem}

\begin{proof}
(i) It follows immediately from the definition of $\Lambda$-transformation.
(ii) It follows from \cite[Lemma 2.1]{CO03} and \cite[Lemma 2.4]{CEO}.
\end{proof}


\begin{lem} \label{lemlambdawell}
Let $b$ be a positive integer, $q$ an odd prime, and $K$ a $\z_q$-lattice such that $b\z_q\subseteq Q(K)\subsetneq \z_q$.
Let $\delta \in \{1,2\}$ be such that $\lambda_q(K)=\Lambda_q(K)^{q^{-\delta}}$.
Then $b\z_q\subseteq Q(\Lambda_q(K))\subseteq q^{\delta}\z_q$.
In particular, $b$ is divisible by $q^{\delta}$ and $bq^{-\delta}\z_q\subseteq Q(\lambda_q(K))$.
\end{lem}

\begin{proof}
Note that the unimodular component in a Jordan decomposition of $K$ is anisotropic since $Q(K)\subsetneq \z_q$.
One may easily deduce from this that
$$
\Lambda_q(K)=\{ \mathbf{x}\in K : Q(\mathbf{x})\in q\z_q \}.
$$
Thus we have $Q(\Lambda_q(K))=Q(K)\cap q\z_q$.
Since $b\z_q\subsetneq \z_q$, it follows that
$$
b\z_q\subseteq Q(K)\cap q\z_q=Q(\Lambda_q(K)).
$$
Note that $Q(\Lambda_q(K))\subseteq \mathfrak{s}(\Lambda_q(K))$ by definition and hence
$$
Q(\Lambda_q(K))\subseteq \mathfrak{s}(\Lambda_q(K))=q^\delta \z_q,
$$
where the equality follows from the choice of $\delta$.
This completes the proof.
\end{proof}


We turn back to the world of triangular forms.
In the remaining of this section, we always let $\mathbf{a}=(a_1,a_2,\dots,a_k)$ be a vector of positive integers.
The triangular form
$$
\E=a_1P_3(x_1)+a_2P_3(x_2)+\cdots+a_kP_3(x_k)
$$
having coefficients vector $\mathbf{a}$ is denoted by $\T(a_1,a_2,\dots,a_k)$ or more simply by $\T(\mathbf{a})$.
Let $R=\z$ or $R=\z_p$ for a prime $p$, and let $n$ be a nonnegative integer.
We say that {\it $n$ is represented by $\E=\T(\mathbf{a})$ over $R$} if the equation
\begin{equation} \label{eqrepn1}
a_1P_3(x_1)+a_2P_3(x_2)+\cdots+a_kP_3(x_k)=n
\end{equation}
is solvable in $R$.
In this case, we write $n\ra \E$ over $R$.
When $R=\z$, we simply say that {\it $n$ is represented by $\E$} and write $n\ra \E$.
We say that {\it $n$ is {\it locally represented by $\E$}} if $n$ is represented by $\E$ over $\z_p$ for all primes $p$.
Denote by $T(\E)$ $\left( T^{\text{loc}}(\E)\right)$ the set of all nonnegative integers represented (locally represented, respectively) by $\E$.
We define $T(\mathbf{a})=T(\E)$ and $T^{\text{loc}}(\mathbf{a})=T^{\text{loc}}(\mathbf{a})$.

Note that if Equation \eqref{eqrepn1} is solvable over $\z$, then it is solvable over $\z_p$ for every prime $p$.
Thus $T(\E) \subseteq T^{\text{loc}}(\E)$ for any triangular form $\E$.
We call $\E$ {\it regular} if $T(\E)=T^{\text{loc}}(\E)$ holds.

Note that $n$ is represented by $\E$ over $R$ if and only if the equation
\begin{equation} \label{eqrepn2}
a_1(2x_1+1)^2+a_2(2x_2+1)^2+\cdots+a_k(2x_k+1)^2=8n+a_1+a_2+\cdots+a_k
\end{equation}
is solvable in $R$.
Unless every $a_i$ is even, one may use Local Square Theorem \cite[63:1]{OM} to see that Equation \eqref{eqrepn2} is always solvable in $\z_2$, as shown in \cite{CO13}.  
For an odd prime $p$, Equation \eqref{eqrepn2} is solvable in $\z_p$ if and only if the equation
$$
a_1y_1^2+a_2y_2^2+\cdots+a_ky_k^2=8n+a_1+a_2+\cdots+a_k
$$
is solvable in $\z_p$ since 2 is a unit in $\z_p$.
For an integer $m$ and a diagonal $\z$-lattice $\langle a_1,a_2,\dots,a_k\rangle$, we write
$$
m\ratwo \langle a_1,a_2,\dots,a_k\rangle \quad (m\nratwo \langle a_1,a_2,\dots,a_k\rangle)
$$
if there is (no, respectively) a vector $(x_1,x_2,\dots,x_k)\in \z^k$ with 
$(2,x_1x_2\cdots x_k)=1$ such that $a_1x_1^2+a_2x_2^2+\cdots+a_kx_k^2=m$. 

We denote by $\mathcal L(\mathbf{a})$ the diagonal $\z$-lattice $\langle a_1,a_2,\dots,a_k\rangle$ and define $DQ(\mathbf{a})$ to be the set of nonnegative integers represented by $\mathcal L(\mathbf{a})$.
For a prime $p$, we let
\begin{align*}
DQ_p(\mathbf{a})&=\{ n\in \z_{\ge 0} : n\ra \mathcal L(\mathbf{a})_p\},\\
\overline{DQ_p(\mathbf{a})}&=\{ \gamma \in \z_p : \gamma \ra \mathcal L(\mathbf{a})_p\}.
\end{align*}
For nonnegative integers $n$, we define
$$
t(n,\mathbf{a})=8n+a_1+a_2+\cdots+a_k.
$$
Under these notations, we get
\begin{align*}
T^{\text{loc}}(\mathbf{a})&=\bigcap_{p:\text{odd prime}} \left\{ n\in \z_{\ge 0} : t(n,\mathbf{a})\in DQ_p(\mathbf{a}) \right\},\\
T(\mathbf{a})&=\left\{ n\in \z_{\ge 0} : t(n,\mathbf{a})\ratwo \mathcal L(\mathbf{a})\right\},
\end{align*}
provided that $\gcd(2,a_1,a_2,\dots,a_k)=1$.

Define $P$ to be the set of all odd primes and we put $P'=\{3,5,7\}$.

\section{Rivers of new regular triangular forms} \label{secriv}

\subsection{New regular triangular forms}

Let $\mathbf{a}=(a_1,a_2,\dots,a_k)$ be a vector of positive integers.
For $i=1,2,\dots,k$, we define
$$
a[i]=(a_1,a_2,\dots,a_{i-1},a_{i+1},a_{i+2},\dots,a_k)\in \n^{k-1}.
$$
For example, if $\mathbf{a}=(3,2,6)\in \n^3$, then we have
$$
\mathbf{a}[1]=(2,6),\ \ \mathbf{a}[2]=(3,6),\ \ \text{and}\ \ \mathbf{a}[3]=(3,2).
$$
A regular $k$-ary triangular form $\T(\mathbf{a})$ is called {\it old} if there is an index $i\in \{1,2,\dots,k\}$ such that $T(\mathbf{a}[i])=T(\mathbf{a})$.
If this is the case, we say that {\it $\T(a_i)$ is redundant in $\T(\mathbf{a})$} (and {\it $\T(a_i)$ is redundant to $\T(\mathbf{a}[i])$}), and we call $\T(a_i)$ {\it a redundant component of $\T(\mathbf{a})$}.
A regular triangular form is called {\it new} if it is not old.

Since the order of the coefficients of a triangular form does matter in some cases such as in the above definition of $\T(\mathbf{a}[i])$, we will be strict with the order of the coefficients of a triangular form.


\begin{lem} \label{lemnew}
Let $k\ge 3$ be an integer and $\mathbf{a}=(a_1,a_2,\dots,a_k)$ be a vector of positive integers with $\gcd(a_1,a_2,\dots,a_k)=1$.
Then the triangular form $\T(\mathbf{a})$ is old regular if and only if there is an index $i\in \{1,2,\dots,k\}$ such that the triangular form $\T(\mathbf{a}[i])$ is primitive regular and $a_i\z_p \subseteq \overline{DQ_p(\mathbf{a}[i])}$ for all odd primes $p$.
\end{lem}

\begin{proof}
We show the ``if" part first.
Assume that $\T(\mathbf{a}[i])$ is primitive regular and $a_i\z_p\subseteq \overline{DQ_p(\mathbf{a}[i])}$ for every odd prime $p$.
Since
$$
T(\mathbf{a}[i])\subseteq T(\mathbf{a})\subseteq T^{\text{loc}}(\mathbf{a})\ \ \text{and}\ \ T^{\text{loc}}(\mathbf{a}[i])=T(\mathbf{a}[i]),
$$
it suffices to show that $T^{\text{loc}}(\mathbf{a})\subseteq T^{\text{loc}}(\mathbf{a}[i])$.
Now, let $n$ be an integer in $T^{\text{loc}}(\mathbf{a})$ and let $p$ be any odd prime.
By Lemma \ref{lemodd1} and the assumption that $a_i\z_p\subseteq \overline{DQ_p(\mathbf{a}[i])}$, we have $\overline{DQ_p(\mathbf{a}[i])}=\overline{DQ_p(\mathbf{a})}$.
Since $n\in T^{\text{loc}}(\mathbf{a})$, we have
$$
t(n,\mathbf{a})\in \overline{DQ_p(\mathbf{a})}=\overline{DQ_p(\mathbf{a}[i])},
$$
or equivalently $t(n,\mathbf{a})\in Q(\mathcal L(\mathbf{a}[i])_p)$.
Hence we have,
$$
t(n,\mathbf{a}[i])=t(n,\mathbf{a})-a_i\in Q(\mathcal L(\mathbf{a}[i])_p \perp \langle -a_i\rangle).
$$
Since $-a_i\z_p=a_i\z_p\subseteq \overline{DQ_p(\mathbf{a}[i])}$,
$$
Q(\mathcal L(\mathbf{a}[i])_p \perp \langle -a_i\rangle)=Q(\mathcal L(\mathbf{a}[i])_p)=\overline{DQ_p(\mathbf{a}[i])}.
$$
by Lemma \ref{lemodd1}.
Hence we have $t(n,\mathbf{a}[i])\in \overline{DQ_p(\mathbf{a}[i])}$.
Since $p$ was an arbitrary odd prime, it follows that $n\in T^{\text{loc}}(\mathbf{a}[i])$.

Conversely, assume that $\T(\mathbf{a})$ is old regular.
Then $T(\mathbf{a})=T(\mathbf{a}[i])$ for some $i$.
To show the primitivity of the triangular form $\T(\mathbf{a}[i])$, we suppose that
there is a prime $\ell$ which divides $a_j$ for every $j\in \{1,2,\dots,k\}-\{i\}$.
Then $a_i$ is not divisible by $\ell$ since $\gcd(a_1,a_2,\dots,a_k)=1$.
Since $T(\mathbf{a}[i])\subseteq \ell \z$,
we have
$$
a_i\in T(\mathbf{a})-T(\mathbf{a}[i]),
$$
which is absurd.
Hence $\T(\mathbf{a}[i])$ is primitive.
The regularity of $\T(\mathbf{a}[i])$ follows from
$$
T(\mathbf{a}[i])\subseteq T^{\text{loc}}(\mathbf{a}[i])\subseteq T^{\text{loc}}(\mathbf{a})=T(\mathbf{a})=T(\mathbf{a}[i]),
$$
where the first two inclusions are obvious and each of the two equalities corresponds to the regularity of $\E$ and the redundancy of $\T(a_i)$, respectively.
Suppose that there is an odd prime $p_1$ such that $a_i\z_{p_1} \not\subseteq \overline{DQ_{p_1}(\mathbf{a}[i])}$.
Define $S$ to be the set of odd primes $q$ such that $\overline{DQ_q(\mathbf{a}[i])}\subsetneq \z_q$.
Since $k\ge 3$, $S$ is a (non-empty) finite set and we may write
$$
S=\{ p_1,p_2,\dots,p_s\}.
$$
Using Lemma \ref{lemodd2}, take $\beta \in \z_{p_1}$ such that
$$
\beta \not\in  \overline{DQ_{p_1}(\mathbf{a}[i])}\ \ \text{and}\ \ \beta+a_i\in \overline{DQ_{p_1}(\mathbf{a})}.
$$
Put
$$
r=1+\max(\ord_p(\beta),\ord_p(\beta+a_i)).
$$
Since $p_1$ is an odd prime, we may take an integer $n_1$ such that
$$
\ord_{p_1}(t(n_1,\mathbf{a}[i])-\beta) \ge r.
$$
For $j=2,3,\dots,s$, we pick an integer $n_j$ such that
$$
t(n_j,\mathbf{a}[i])\equiv 0\Mod {p_j^{\ord_{p_j}(a_i)+1}}.
$$
By Chinese Remainder Theorem, there is a positive integer $n$ satisfying
$$
n\equiv n_1\Mod{p_1^r}\ \ \text{and}\ \ n\equiv n_j \Mod {p_j^{\ord_{p_j}(a_i)+1}},\ \ j=2,\dots,s.
$$
Then one may easily check that
$$
n\in T^{\text{loc}}(\mathbf{a})-T^{\text{loc}}(\mathbf{a}[i]).
$$
From the regularity of $\T(\mathbf{a})$, it follows that
$$
T^{\text{loc}}(\mathbf{a})-T^{\text{loc}}(\mathbf{a}[i])=T(\mathbf{a})-T^{\text{loc}}(\mathbf{a}[i])\subseteq T(\mathbf{a})-T(\mathbf{a}[i]).
$$
Hence we have $n\in T(\mathbf{a})-T(\mathbf{a}[i])$, which contradicts to the redundancy of $\T(a_i)$.
This completes the proof.
\end{proof}


\begin{rmk} \label{rmknew}
Let $\T(\mathbf{a})=\T(a_1,a_2,\dots,a_k)$ be an old regular triangular form of rank $k\ge 4$ with $\T(a_i)$ redundant in it.
Then, by Lemma \ref{lemnew}, the quadratic $\q_p$-space $\q_p\mathcal{L}(\mathbf{a}[i])$ is universal for all odd primes $p$.
\end{rmk}


\begin{rmk} \label{rmkdivisor}
Let $\mathbf{a}=(a_1,a_2,\dots,a_k)\in \n^k$ be a vector such that $\q_p\mathcal{L}(\mathbf{a})$ is universal for every odd prime $p$.
For each odd prime $p$, we define a nonnegative integer $e_p$ to be the smallest integer satisfying $p^{e_p}\z_p\subseteq \overline{DQ_p(\mathbf{a})}$.
Note that $e_p$ is a nonnegative integer for any odd prime $p$, and $e_p=0$ for all but finitely many odd primes $p$.
We define $\xi(\mathbf{a})=\prod_p p^{e_p}$, where the product runs over all odd primes.
\end{rmk}


\begin{prop} \label{propter}
Any ternary regular triangular form is new.
\end{prop}

\begin{proof}
It follows immediately from Remark \ref{rmknew} and the fact that every binary quadratic space over $\q$ is anisotropic at infinitely many primes.
\end{proof}

\subsection{Watson transformations} \label{subsecwatson}

Let $p$ be an odd prime and let $\mathbf{a}=(a_1,a_2,\dots,a_k)\in \n^k$ with $\gcd(a_1,a_2,\dots,a_k)=1$.
For $i=1,2,\dots,k$, we let
$$
s_i=\begin{cases}0&\text{if}\ \ \ord_p(a_i)\ge 1,\\
2&\text{otherwise},\end{cases}
$$
and put
$$
s=\min \{ \ord_p(p^{s_i}a_i) : 1\le i\le k\} \in \{1,2\}.
$$
Then one may easily check that
\begin{align*}
\Lambda_p(\mathcal L(\mathbf{a}))&=\langle p^{s_1}a_1,p^{s_2}a_2,\dots,p^{s_k}a_k\rangle,\\
\lambda_p(\mathcal L(\mathbf{a}))&=\langle p^{s_1-s}a_1,p^{s_2-s}a_2,\dots,p^{s_k-s}a_k\rangle.
\end{align*}
For compatibility, we define
\begin{align*}
\Lambda_p(\T(\mathbf{a}))&=\T(p^{s_1}a_1,p^{s_2}a_2,\dots,p^{s_k}a_k),\\
\lambda_p(\T(\mathbf{a}))&=\T(p^{s_1-s}a_1,p^{s_2-s}a_2,\dots,p^{s_k-s}a_k).
\end{align*}
In practice, we also need the following notations;
\begin{align*}
\Lambda_p(\mathbf{a})&=(p^{s_1}a_1,p^{s_2}a_2,\dots,p^{s_k}a_k),\\
\lambda_p(\mathbf{a})&=(p^{s_1-s}a_1,p^{s_2-s}a_2,\dots,p^{s_k-s}a_k),\\
\lambda_p^{-1}(\mathbf{a})&=\{ \mathbf{b}=(b_1,b_2,\dots,b_k)\in \n^k : \gcd(b_1,b_2,\dots,b_k)=1,\ \lambda_p(\mathbf{b})=\mathbf{a}\}.
\end{align*}

We briefly look at regular ``ternary" triangular forms.
Let $p$ be an odd prime and let $\mathbf{b}=(b_1,b_2,b_3)\in \n^3$.
Put $\E=\T(\mathbf{b})$.
We say that {\it $\E$ is $p$-stable} if
$$
\langle 1,-1\rangle \ra \mathcal{L}(\mathbf{b}) \quad \text{or}\quad \mathcal{L}(\mathbf{b})\simeq \langle 1,-\Delta_p\rangle \perp \langle p\epsilon_p\rangle,
$$
where $\epsilon_p\in \z_p^{\times}$, and {\it $p$-unstable} otherwise.
If $\E$ is $p$-stable for every odd prime $p$, then $\E$ is called {\it stable}, and is called {\it unstable} otherwise.
Note that $\lambda_p$-transformation preserves the regularity of a ternary triangular form, more precisely, if $\E$ is regular and $p$-unstable, then $\lambda_p(\E)$ is also regular (for this, see \cite[Lemmas 2.1 and 2.2]{KO1}).
Hence by taking $\lambda_q$-transformations for some odd primes $q$ to a given (unstable) regular ternary triangular form $\T(b_1,b_2,b_3)$, we always get a stable regular ternary triangular form $\T(b_1',b_2',b_3')$.
In \cite{KO1}, we found all regular ternary triangular forms by finding the stable regular ones first and tracing the $\lambda_q$-transformations back.


Let $\mathbf{a}\in \n^k$ be a vector with $k\ge 4$ such that the $k$-ary triangular form $\E=\T(\mathbf{a})$ is regular.
For an odd prime $p$, we call $\E$ {\it $p$-stable} if $\overline{DQ_p(\mathbf{a})}=\z_p$, and {\it $p$-unstable} otherwise.
We call $\E$ {\it stable} if it is $p$-stable for all odd primes $p$.
It is obvious from definition that a regular triangular form $\E$ of rank $\ge 4$ is stable if and only if it is universal.
Nevertheless, we shall use the notion of stableness for regular triangular forms of rank $\ge 4$ for compatibility with ternary cases.


\begin{lem} \label{lemreci}
Let $p$ be an odd prime and let $\mathbf{a}=(a_1,p^{e_2}a_2,p^{e_3}a_3,\dots,p^{e_k}a_k)$ be a vector with $0\le e_2\le e_3\le \cdots \le e_k$ such that $(p,a_1a_2\cdots a_k)=1$ and $\gcd(a_1,p^{e_2}a_2,p^{e_3}a_3,\dots,p^{e_k}a_k)=1$.
Assume that $\T(\mathbf{a})$ is $p$-unstable.
We put $\beta$ to be the sum of all coefficients of $\mathbf{a}$ coprime to $p$.
Then for any nonnegative integer $n$, we have the following:
\begin{enumerate} [(i)]
\item If $\T(\mathbf{a})$ is regular and $n\in T^{\text{loc}}(\Lambda_p(\mathbf{a}))$, then $n\in T(\Lambda_p(\mathbf{a}))$.
\item If $\Lambda_p(\T(\mathbf{a}))$ is regular, $n\in T^{\text{loc}}(\mathbf{a})$, $t(n,\mathbf{a})\equiv 0\Mod p$, and $n\ge \dfrac{p^2-1}{8}\beta$, then $n\in T(\mathbf{a})$.
\end{enumerate}
\end{lem}

\begin{proof}
Since $\T(\mathbf{a})$ is $p$-unstable, either $1\le e_2$, or $0=e_2<e_3$ and $\left(\dfrac{-a_1a_2}{p}\right)=-1$ holds.
We provide the proof only for the latter case ($0=e_2<e_3$) since the first case ($1\le e_2$) may be dealt with in a similar manner.

\noindent (i) Note that
$$
\Lambda_p(\mathbf{a})=(p^2a_1,p^2a_2,p^{e_3}a_3,p^{e_4}a_4,\dots,p^{e_k}a_k).
$$
Since $n\in T^{\text{loc}}(\Lambda_p(\mathbf{a}))$, we have
$$
t(n,\Lambda_p(\mathbf{a}))\in \overline{DQ_p(\Lambda_p(\mathbf{a}))}\ \ \text{for all}\ \ p\in P.
$$
In particular,
$$
t(n,\Lambda_p(\mathbf{a}))\in \overline{DQ_p(\Lambda_p(\mathbf{a}))}\subset p\z_p,
$$
and thus $t(n,\Lambda_p(\mathbf{a}))$ is divisible by $p$.
Set
$$
n'=n+\frac{p^2-1}{8}\beta,
$$
where we have $\beta=a_1+a_2$.
One may easily check that
$$
t(n',\mathbf{a})=t(n,\Lambda_p(\mathbf{a})).
$$
This implies that
$$
t(n',\mathbf{a})=t(n,\Lambda_p(\mathbf{a}))\in \overline{DQ_q(\Lambda_p(\mathbf{a}))} \subset \overline{DQ_q(\mathbf{a})}\ \ \text{for all}\ \ q\in P.
$$
From this and the regularity of $\T(\mathbf{a})$, it follows that there is a vector $(x_1,x_2,\dots,x_k)\in \z^k$ with $(2,x_1x_2\cdots x_k)=1$ such that
$$
t(n',\mathbf{a})=a_1x_1^2+a_2x_2^2+p^{e_3}a_3x_3^2+\cdots+p^{e_k}a_kx_k^2.
$$
Since $t(n',\mathbf{a})\equiv 0\Mod p$ and $\left(\dfrac{-a_1a_2}{p}\right)=-1$, we have $x_1\equiv x_2\equiv 0\Mod p$.
If we write $x_i=py_i$ for $i=1,2$, then we have
\begin{align*}
t(n,\Lambda_p(\mathbf{a}))&=t(n',\mathbf{a})\\
&=p^2a_1y_1^2+p^2a_2y_2^2+p^{e_3}a_3x_3^2+\cdots+p^{e_k}a_kx_k^2,
\end{align*}
where $\left(2,y_1y_2x_3\cdots x_k\right)=1$.
Hence we have $n\in T(\Lambda_p(\mathbf{a}))$.

\noindent (ii) Since $n\in T^{\text{loc}}(\mathbf{a})$, we have $t(n,\mathbf{a})\in \overline{DQ_q(\mathbf{a})}$ for all $q\in P$.
In particular, 
$$
t(n,\mathbf{a})\in p\z_p\cap \overline{DQ_p(\mathbf{a})}=\overline{DQ_p(\Lambda_p(\mathbf{a}))},
$$
where the equality comes from Lemma \ref{lemqzq}(ii).
It follows that
$$
t(n,\mathbf{a})\in \overline{DQ_q(\Lambda_p(\mathbf{a}))}\ \ \text{for all}\ \ q\in P,
$$
since $\Lambda_p(\mathcal L(\mathbf{a}))_q=\mathcal L(\mathbf{a})_q$ for all $q\in P-\{p\}$.
Put
$$
n'=n-\frac{p^2-1}{8}\beta.
$$
Then $n'\ge 0$ by the assumption that $n\ge \dfrac{p^2-1}{8}\beta$, and one may easily check that $t(n',\Lambda_p(\mathbf{a}))=t(n,\mathbf{a})$.
Thus
$$
t(n',\Lambda_p(\mathbf{a}))=t(n,\mathbf{a})\in \overline{DQ_q(\Lambda_p(\mathbf{a}))}\ \ \text{for all}\ \ q\in P.
$$
From this and the regularity of $\Lambda_p(\T(\mathbf{a}))$, there is a vector $(z_1,z_2,\dots,z_k)\in \z^k$ with $(2,z_1z_2\cdots z_k)=1$ such that
$$
t(n',\Lambda_p(\mathbf{a}))=p^2a_1z_1^2+p^2a_2z_2^2+p^{e_3}a_3z_3^2+\cdots+p^{e_k}a_kz_k^2.
$$
Hence we have
$$
t(n,\mathbf{a})=t(n',\Lambda_p(\mathbf{a}))=a_1(pz_1)^2+a_2(pz_2)^2+p^{e_3}a_3z_3^2+\cdots+p^{e_k}a_kz_k^2,
$$
where $(2,pz_1\cdot pz_2\cdot z_3z_4\cdots z_k)=1$.
Thus $n\in T(\mathbf{a})$.
This completes the proof.
\end{proof}


\begin{prop} \label{propdescend}
Let $p$ be an odd prime and let $\E$ be a $p$-unstable regular triangular form of rank $\ge 4$.
Then $\lambda_p(\E)$ is also regular. 
\end{prop}

\begin{proof}
Let $n\in T^{\text{loc}}(\lambda_p(\E))$.
Then $sn\in T^{\text{loc}}(\Lambda_p(\E))$, where $s$ is the integer satisfying $\lambda_p(\E)=\Lambda_p(\E)^s$.
By Lemma \ref{lemreci}(i), we have $sn\in T(\Lambda_p(\E))$.
Hence $n\in T(\lambda_p(\E))$.
This completes the proof.
\end{proof}

\newpage

\begin{lem} \label{lemlocrep}
Let $p$ be an odd prime, $u$ an integer with $(u,p)=1$, $v$ an arbitrary integer, and $\epsilon_i\in \z_p^{\times}$ for $i=1,2,3$.
Then we have
\begin{enumerate} [(i)]
\item $\left\vert \left\{ n : un+v\ra \langle \epsilon_1\rangle,\ 0\le n\le p-1\right\} \right\vert \ge \dfrac{p-1}{2}$,\\
\item $\left\vert \left\{ n : un+v\ra \langle \epsilon_1,\epsilon_2\rangle,\ 0\le n\le p-1\right\} \right\vert \ge p-1$,\\
\item $\left\vert \left\{ n : un+v\ra \langle \epsilon_1,p\epsilon_2\rangle,\ 0\le n\le p^2-1\right\} \right\vert \ge \dfrac{p^2-1}{2}$,\\
\item $\left\vert \left\{ n : un+v\ra \langle \epsilon_1,p\epsilon_2,p\epsilon_3\rangle,\ 0\le n\le p^2-1\right\} \right\vert \ge \dfrac{p^2+p-2}{2}$.
\end{enumerate}
\end{lem}

\begin{proof}
This is quite straightforward and left as an exercise to the reader.
\end{proof}


\begin{lem} \label{lemprimes}
Any primitive regular triangular form $\E$ of rank $\ge 4$ is $p$-stable for every prime $p$ greater than 7.
\end{lem}

\begin{proof}
Let $p$ be an odd prime greater than 7 and let
$$
\E=\T(a_1,p^{e_2}a_2,p^{e_3}a_3,\dots,p^{e_k}a_k)
$$
be a primitive regular triangular form of rank $k\ge 4$ with $0\le e_2\le e_3\le \cdots \le e_k$ and $(p,a_1a_2\cdots a_k)=1$.
For simplicity, put
$$
\mathbf{a}=(a_1,p^{e_2}a_2,p^{e_3}a_3,\dots,p^{e_k}a_k)\ \ \text{and}\ \ \alpha=a_1+p^{e_2}a_2+p^{e_3}a_3+\cdots+p^{e_k}a_k.
$$

Suppose that $\E$ is $p$-unstable, i.e., $\overline{DQ_p(\mathbf{a})}\subsetneq \z_p$.
Using Proposition \ref{propdescend}, we may assume $\overline{DQ_p(\mathbf{a})}=\z_q$ for all odd primes $q\neq p$ by taking $\lambda_q$-transformations for odd primes $q\neq p$, if necessary.
Then we have
$$
T^{\text{loc}}(\mathbf{a})=\{ n\in \z_{\ge 0} : t(n,\mathbf{a})\in \overline{DQ_p(\mathbf{a})} \}.
$$
From this and the regularity of $\T(\mathbf{a})$, it follows that
\begin{equation} \label{eqlemprimes1}
T(\mathbf{a})=\{ n\in \z_{\ge 0} : t(n,\mathbf{a})\in \overline{DQ_p(\mathbf{a})} \}.
\end{equation}

First, assume that $e_2=0$.
Since $\overline{DQ_p(\mathbf{a})}\subsetneq \z_p$, it follows that $e_3>0$ and $\left( \dfrac{-a_1a_2}{p}\right)=-1$.
Without loss of generality, we may assume that $a_1\le a_2$.
Suppose that $a_1=1$ and $a_2\ge 5$.
Then we have
$$
n\not\in T(a_1,a_2,p^{e_3}a_3,\dots,p^{e_k}a_k)
$$
for $n=2,4$.
From this and Equation \eqref{eqlemprimes1}, we have 
$$
t(n,\mathbf{a})\not\in \overline{DQ_p(\mathbf{a})}
$$
for $n=2,4$.
Note that
$$
\z_p^{\times}=\overline{DQ_p(a_1,a_2)}\subset \overline{DQ_p(\mathbf{a})}
$$
by \cite[92:1b]{OM}.
Hence we have
$$
t(n,\mathbf{a})\equiv 0\Mod p,\quad n=2,4.
$$
This is absurd since $p$ cannot divide $16=t(4,\mathbf{a})-t(2,\mathbf{a}).$
One may use
$$
\begin{cases}
n=1,4&\text{if}\ \ a_1=2,\ a_2=3,\\
n=1,3&\text{if}\ \ a_1=2,\ a_2\ge 4,\\
n=1,2&\text{if}\ \ a_1\ge 3
\end{cases}
$$
to show the absurdity in a similar manner.
Suppose that $(a_1,a_2)=(1,1)$.
Since
$$
5\not\in T(1,1,p^{e_3}a_3,\dots,p^{e_k}a_k),
$$
one may easily deduce, by a similar reasoning as above, that
$$
8\cdot 5+1+1+p^{e_3}a_3+\cdots+p^{e_k}a_k\equiv 0\Mod p.
$$
This implies that $p$ divides $8\cdot 5+1+1=42$, which is absurd since $p>7$.
One may use
$$
\begin{cases}
n=4&\text{if}\ \ (a_1,a_2)=(1,2),\\
n=2&\text{if}\ \ (a_1,a_2)\in \{(1,3),(1,4)\},\\
n=1&\text{if}\ \ (a_1,a_2)=(2,2),
\end{cases}
$$
to deduce the absurdity.

Second, assume that $e_2>0$.
We define a set $G$ by
$$
G=\left\{ 0\le n\le p-1 : n\in  T(a_1) \right\}
$$
and put $g=\vert G\vert$.
Since
$$
G=\left\{ a_1\frac{h(h+1)}{2} : 0\le h\le g-1\right\},
$$
we have
$$
a_1\frac{(g-1)g}{2}\le p-1.
$$
From this, one may easily deduce that
$$
(2g-1)^2\le 1+\frac{8p-8}{a_1}\le 8p-7
$$
and thus we have 
\begin{equation} \label{eqlemprimes2}
g\le \dfrac{\sqrt{8p-7}+1}{2}.
\end{equation}
On the other hand,
\begin{align*}
G&=\left\{ 0\le n\le p-1 : n\in T(a_1) \right\} \\
&=\left\{ 0\le n\le p-1 : n\in T(a_1,p^{e_2}a_2,\cdots,p^{e_k}a_k) \right\} \\
&=\left\{ 0\le n\le p-1 : t(n,\mathbf{a})\in \overline{DQ_p(\mathbf{a})} \right\},
\end{align*}
where the second equality is obvious and the third follows from Equation \eqref{eqlemprimes1}.
Now we have
$$
g=\left\vert \left\{ 0\le n\le p-1 : t(n,\mathbf{a})\in \overline{DQ_p(\mathbf{a})} \right\} \right\vert \ge \frac{p-1}{2}
$$
by Lemma \ref{lemlocrep}(i) with $u=8$ and $v=\alpha$.
From this and \eqref{eqlemprimes2} follows that
$$
\frac{p-1}{2}\le g\le \frac{\sqrt{8p-7}+1}{2}.
$$
Thus we have $p\le 11$.
Suppose that $p=11$.
By using Lemma \ref{lemlocrep}(i) with $p=11$, $u=8$ and $v=\alpha$, we have
$$
\left\vert \left\{ 0\le n\le 32 : t(n,\mathbf{a})\in \overline{DQ_p(\mathbf{a})} \right\} \right\vert \ge 15.
$$
From this and Equation \eqref{eqlemprimes1}, it follows that
$$
\left\vert \left\{ 0\le n\le 32 : n\in T(a_1,11^{e_2}a_2,\dots,11^{e_k}a_k) \right\} \right\vert \ge 15.
$$
From this and the fact that
$$
\left\{ 0\le n\le 32 : n\in T(11^{e_2}a_2,\dots,11^{e_k}a_k) \right\} \subseteq \{0,11,22\},
$$
one may easily deduce that $a_1=1$ and $e_2=e_3=1$.
Thus we may write
$$
\E=\T(1,11a_2,11a_3,11^{e_4}a_4,\dots,11^{e_k}a_k).
$$
If $e_4=1$, then we have $11\z_{11}\subset \overline{DQ_{11}(\mathbf{a})}$.
Since
$$
8\cdot 4+1+11a_2+11a_3+11a_4+\cdots+11^{e_k}a_k\equiv 0\Mod {11},
$$
we have $4\in T^{\text{loc}}(\mathbf{a})$.
This is absurd since $\T(\mathbf{a})$ cannot represent 4.
So we may further assume that $e_4\ge 2$.
By Lemma \ref{lemlocrep}(iv), we have
$$
\left\vert \left\{ 0\le n\le 120 : n\in T(1,11a_2,11a_3,11^{e_4}a_4,\dots,11^{e_k}a_k) \right\} \right\vert \ge 65.
$$
However, one may easily check that
$$
\left\vert \left\{ 0\le n\le 120 : n\in T(1,11a_2,11a_3) \right\} \right\vert \le 63
$$
for all pairs $(a_2,a_3)$ satisfying $1\le a_2\le a_3\le 10$.
This contradicts to the regularity of $\T(\mathbf{a})$.
This completes the proof.
\end{proof}

\subsection{New stable regular triangular forms} \label{subsecnews}

In this short subsection, we gather facts on new stable regular triangular forms of rank $\ge 3$.
Recall that any regular ternary triangular form is new (see Proposition \ref{propter}).


\begin{prop}
There are exactly 17 (new) stable regular ternary triangular forms, and they are marked with $\dag$ in Table \ref{tablerk3}.
\end{prop}

\begin{proof}
See \cite[Theorem 3.17]{KO1}.
\end{proof}


Let $\E$ be a new stable regular triangular form of rank $\ge 4$.
Then $\E$ is a universal triangular form that is new.
The following result is well known and we record it here as a proposition for later reference.


\begin{prop}
There are exactly four new stable regular triangular forms of rank exceeding 3, and they are $\T(1,1,3,a_4)$ for $a_4\in \{3,6,7,8\}$.
\end{prop}

\begin{proof}
See the proof of \cite[Theorem 2.1]{BK}.
\end{proof}

\subsection{Drop structures} \label{subsecdrop}

Let $\E$ be an unstable regular $k$-ary triangular form with $k\ge 4$.
By Proposition \ref{propdescend} and Lemma \ref{lemprimes}, there is a chain of regular $k$-ary triangular forms
\begin{equation} \label{eqriver1}
\E=\E_1 \overset{\lambda_{p_1}}{\ra} \E_2 \overset{\lambda_{p_2}}{\ra} \E_3 \overset{\lambda_{p_3}}{\ra} \cdots \overset{\lambda_{p_{\nu-1}}}{\ra} \E_\nu
\end{equation}
where $\lambda_{p_i}(\E_i)=\E_{i+1}$, $p_i\in P'=\{3,5,7\}$,  $\E_i$ is $p_i$-unstable for $i=1,2,\dots,\nu-1$, and $\E_{\nu}$ is stable.
Let's assume further that $\E$ is new.
If we take newness into account, obviously there are two possibilities;
\begin{enumerate} [(I)]
\item $\E_i$ is new for every $i\in \{1,2,\dots,\nu \}$;
\item $\E_j$ is old for some $j$.
\end{enumerate}
Note that (II) is always the case when $k\ge 5$ since there does not exist a new stable regular triangular form of rank exceeding 4.
When $k=4$, both cases can happen and if (II) is the case, then there is a chain of regular triangular forms
\begin{align*}
\E=\E_1 \overset{\lambda_{p_1}}{\ra} \E_2 \overset{\lambda_{p_2}}{\ra} \cdots \overset{\lambda_{p_{j-2}}}{\ra} \E_{j-1} \overset{\lambda_{p_{j-1}}}{\ra}&\E_j \ \left(\overset{\lambda_{p_j}}{\ra}\E_{j+1}\overset{\lambda_{p_{j+1}}}{\ra}\E_{j+1}\cdots \overset{\lambda_{p_{\nu-1}}}{\ra}\E_{\nu}\right)\\
&\big\downarrow \\
&\E'_{j+1} \overset{\lambda_{q_1}}{\ra} \E'_{j+2} \overset{\lambda_{q_2}}{\ra} \cdots \overset{\lambda_{q_{\mu-1}}}{\ra} \E'_{j+\mu}
\end{align*}
where
\begin{enumerate} [(i)]
\item $\E_i$ is unstable new regular quaternary for $i=1,2,\dots,j-1$,
\item $\E_j=\T(\mathbf{b})=\T(b_1,b_2,b_3,b_4)$ is old regular quaternary with $\T(b_u)$ redundant,
\item $\E'_{j+1}=\T(\mathbf{b}[u])$ is regular ternary,
\item $\E'_{j+m}$ is unstable regular ternary for $m=1,2,\dots,\mu-1$,
\item $\E'_{j+\mu}$ is stable regular ternary.
\end{enumerate}
Note that $\mu$ could be equal to 1, i.e., it is possible that $\E'_{j+1}$ is stable.

Genearlly, if $\E$ is a $p$-unstable new regular $k$-ary triangular form with
$$
\lambda_p(\E)=\T(c_1,c_2,\dots,c_k)
$$
is old such that there are $r$ ($r\in \n$) indices $u_1,\dots,u_r$ with $1\le u_1<\cdots<u_r\le k$ such that
$$
\E'=\T(c_1,\dots,c_{i_1-1},\widehat{c_{i_1}},c_{i_1+1},\dots,c_{i_2-1},\widehat{c_{i_2}},c_{i_2+1},\dots,\widehat{c_{i_r}},\dots,c_k),
$$
where the symbol $\ \widehat{c_j}\ $ means that the component $c_j$ is missing, is a new regular $(k-r)$-ary triangular form, then we call the following diagram
\begin{align*}
\E \overset{\lambda_p}{\ra} \lambda_p&(\E)\\
&\big\downarrow \\
&\E'
\end{align*}
{\it a drop structure on rank $k$}.
Let this drop structure be denoted by $DS$.
Here, we call $\E$ {\it the top of $DS$} and $\E'$ {\it the bottom of $DS$}.
We define {\it the height of $DS$} to be $r$.
We denote the above drop structure by $\E \overset{\lambda_p}{\ltra} \E'$ also.

\subsection{Rivers of new regular triangular forms} \label{subsecriv}

As we saw in \eqref{eqriver1}, any (not necessarily new) regular triangular form, via $\lambda_p$-transformations, flows into a stable regular form of the same rank.
For a new regular triangular form, if necessary, by building drop structures, one may lower the water level (i.e. rank) and ultimately make it flow into a new stable regular triangular form.
For each new stable regular triangular form $\E$, starting from $\E$, we trace back the $\lambda_p$-transformations for $p\in P'$ and inspect the drop structures to obtain a river of new regular triangular forms which has $\E$ as its river mouth.
An example is given in Table \ref{river112}.


\begin{table}
\caption{The river of new regular triangular forms with mouth $\T(1,1,2)$}
{\begin{tikzpicture}
      [
      arrow1/.style = {draw, -latex},
      ]

      \coordinate (A1) at (0cm, 2cm);

      \coordinate (B1) at (1.75cm, 2cm);

      \coordinate (C1) at (3.75cm, 2cm);
      
      \coordinate (D7) at (5.75cm, 11cm);
      \coordinate (D6) at (5.75cm, 9.5cm);
      \coordinate (D5) at (5.75cm, 8cm);
      \coordinate (D4) at (5.75cm, 6.5cm);
      \coordinate (D3) at (5.75cm, 5cm);
      \coordinate (D2) at (5.75cm, 3.5cm);
      \coordinate (D1) at (5.75cm, 2cm);
      
      \coordinate (E5) at (8.25cm, 8.75cm);
      \coordinate (E3) at (8.25cm, 5.75cm);
      \coordinate (E1) at (8.25cm, 2.75cm);
      
      \coordinate (F1) at (10.75cm, 3.5cm);

      \coordinate (O) at (4.5cm, 0cm);

      \node at (A1) {$\T(1,1,18)$};

      \node at (B1) {$\T(1,9,18)$};

      \node at (C1) {$\T(1,7,7,14)$};
      
      \node at (D7) {$\vdots$};
      \node at (D6) {$\T(1,1,6,1458)$};
      \node at (D5) {$\T(2,3,3,486)$};
      \node at (D4) {$\T(1,1,6,162)$};
      \node at (D3) {$\T(2,3,3,54)$};
      \node at (D2) {$\T(1,1,6,18)$};
      \node at (D1) {$\T(2,3,3,6)$};
      
      \node at (E5) {$\T(1,6,9,1458)$};
      \node at (E3) {$\T(1,6,9,162)$};
      \node at (E1) {$\T(1,6,9,18)$};
      
      \node at (F1) {$\T(2,3,27,54)$};

      \node at (O) {$\T(1,1,2)$};

      \path [arrow1] ([yshift=-3mm]A1) -- node [below][xshift=-1.5mm,yshift=0.5mm] {$\lambda_{3}$} ([xshift=-6mm,yshift=3mm]O);
      
      \path [arrow1] ([yshift=-3mm]B1) -- node [below][xshift=-2mm,yshift=2.25mm] {$\lambda_{3}$} ([xshift=-2mm,yshift=3mm]O);
      
      \draw [-{latex}{latex}] ([yshift=-3mm]C1) -- node [below][xshift=3.25mm,yshift=4mm] {$\lambda_{7}$} ([xshift=2mm,yshift=3mm]O);
      
      \draw [-{latex}{latex}] ([yshift=-3mm]D1) -- node [below][xshift=3mm,yshift=2mm] {$\lambda_{3}$} ([xshift=6mm,yshift=3mm]O);
      
      \path [arrow1] ([yshift=-3mm]D2) -- node [below][xshift=3mm,yshift=3mm] {$\lambda_{3}$} ([yshift=3mm]D1);
      \path [arrow1] ([yshift=-3mm]D3) -- node [below][xshift=3mm,yshift=3mm] {$\lambda_{3}$} ([yshift=3mm]D2);
      \path [arrow1] ([yshift=-3mm]D4) -- node [below][xshift=3mm,yshift=3mm] {$\lambda_{3}$} ([yshift=3mm]D3);
      \path [arrow1] ([yshift=-3mm]D5) -- node [below][xshift=3mm,yshift=3mm] {$\lambda_{3}$} ([yshift=3mm]D4);
      \path [arrow1] ([yshift=-3mm]D6) -- node [below][xshift=3mm,yshift=3mm] {$\lambda_{3}$} ([yshift=3mm]D5);
       \path [arrow1] ([yshift=-4mm]D7) -- node [below][xshift=3mm,yshift=3mm] {$\lambda_{3}$} ([yshift=3mm]D6);
       
      \path [arrow1] ([xshift=-10mm,yshift=-1.5mm]E1) -- node [below][xshift=2.5mm,yshift=1.5mm] {$\lambda_{3}$} ([xshift=8mm,yshift=2mm]D1);
      \path [arrow1] ([xshift=-11mm,yshift=-1.5mm]E3) -- node [below][xshift=2.5mm,yshift=1.5mm] {$\lambda_{3}$} ([xshift=8mm,yshift=2mm]D3);
      \path [arrow1] ([xshift=-11.5mm,yshift=-1.5mm]E5) -- node [below][xshift=2.5mm,yshift=1.5mm] {$\lambda_{3}$} ([xshift=9mm,yshift=2mm]D5);
      
       \path [arrow1] ([xshift=-10mm,yshift=-1.5mm]F1) -- node [below][xshift=2.5mm,yshift=1.5mm] {$\lambda_{3}$} ([xshift=8mm,yshift=2mm]E1);
      
\end{tikzpicture}}
\label{river112}
\end{table}


Let's take a look at the anatomy of the river appearing in Table \ref{river112}.
Its mouth is $\T(1,1,2)$.
In this river, there is exactly one main stream containing the forms $\T(2,3,3,2\cdot 3^{2r-1})$ and $\T(1,1,6,2\cdot 3^{2r})$ for $r\in \n$.
Upon the main stream, there are periodic tributaries with confluences $\T(2,3,3,2\cdot 3^{2r-1})$ for $r\ge 2$ and sources $\T(1,6,9,2\cdot 3^{2r})$.
Only the tributary with confluence $\T(2,3,3,6)$ contains a sporadic form $\T(2,3,27,54)$ as its source.
There are exactly two drop structures built in this river as given in \eqref{drop112}.
\begin{align}  \label{drop112}
\T(1,7,7,14) \overset{\lambda_7}{\ra} \T(1,1&,2,7)\quad \quad & \T(2,3,3,6) \overset{\lambda_3}{\ra} \T(1,1&,2,6)\nonumber \\
&\big\downarrow & &\big\downarrow \\
\T(1&,1,2) & \T(1&,1,2) \nonumber
\end{align}

\section{New regular quaternary triangular forms} \label{secqua}

As mentioned, there are exactly four new stable regular quaternary triangular forms.
Throughout this section, they are denoted as following:
$$
\E_1=\T(1,1,3,3),\ \ \E_2=\T(1,1,3,6),\ \ \E_3=\T(1,1,3,7)\ \ \text{and}\ \ \E_4=\T(1,1,3,8).
$$
At first, we establish the regularity of every quaternary triangular form in Table \ref{tablerk4}.


\begin{table}[ht]
\caption{New regular quaternary triangular forms $\T(a_1,a_2,a_3,a_4)$ ($r\in \n$)}
\label{tablerk4}
\begin{tabular}{|lll|l|}
\hline \rule[-2mm]{0mm}{6mm}
$a_1$ & $a_2$ & $a_3$ & $a_4$\\
\hline \rule[-2mm]{0mm}{6mm}
1 & 1 & 3 & $3^r,\ 2\cdot 3^r,\ 4\cdot 3^{2r},\ 5\cdot 3^{2r},\ 7\cdot 3^{2r-2},\ 8\cdot 3^{2r-2}$\\
\hline \rule[-2mm]{0mm}{6mm}
1 & 1 & 6 & $3^{r+1},\ 2\cdot 3^{2r},\ 4\cdot 3^{2r},\ 5\cdot 3^{2r}$\\
\hline \rule[-2mm]{0mm}{6mm}
1 & 2 & 5 & $5^{2r},\ 2\cdot 5^{2r},\ 3\cdot 5^{2r},\ 4\cdot 5^{2r}$\\
\hline \rule[-2mm]{0mm}{6mm}
1 & 3 & 3 & $3^r,\ 2\cdot 3^r,\ 4\cdot 3^{2r-1},\ 5\cdot 3^{2r-1},\ 7\cdot 3^{2r-1},\ 8\cdot 3^{2r-1}$\\
\hline \rule[-2mm]{0mm}{6mm}
1 & 3 & 4 & $3^{2r},\ 2\cdot 3^{2r}$\\
\hline \rule[-2mm]{0mm}{6mm}
1 & 3 & 9 & $3^{r+1},\ 2\cdot 3^{r+1},\ 4\cdot 3^{2r},\ 5\cdot 3^{2r},\ 7\cdot 3^{2r},\ 8\cdot 3^{2r}$\\
\hline \rule[-2mm]{0mm}{6mm}
1 & 3 & 12 & $3^{2r+1},\ 2\cdot 3^{2r+1}$\\
\hline \rule[-2mm]{0mm}{6mm}
1 & 3 & 27 & $3^{r+2},\ 2\cdot 3^{r+2},\ 4\cdot 3^{2r+1},\ 5\cdot 3^{2r+1},\ 7\cdot 3^{2r+1},\ 8\cdot 3^{2r+1}$\\
\hline \rule[-2mm]{0mm}{6mm}
1 & 5 & 10 & $5^{2r+1},\ 2\cdot 5^{2r-1},\ 3\cdot 5^{2r-1},\ 4\cdot 5^{2r-1}$\\
\hline \rule[-2mm]{0mm}{6mm}
1 & 6 & 6 & $3^{2r+1}$\\
\hline \rule[-2mm]{0mm}{6mm}
1 & 6 & 9 & $3^{r+1},\ 2\cdot 3^{2r},\ 4\cdot 3^{2r},\ 5\cdot 3^{2r}$\\
\hline \rule[-2mm]{0mm}{6mm}
1 & 6 & 18 & $18,27,36$\\
\hline \rule[-2mm]{0mm}{6mm}
1 & 7 & 7 & $7,14,21,28,35$\\
\hline \rule[-2mm]{0mm}{6mm}
1 & 7 & 14 & $14,21,28$\\
\hline \rule[-2mm]{0mm}{6mm}
2 & 2 & 3 & $3^{2r}$\\
\hline \rule[-2mm]{0mm}{6mm}
2 & 3 & 3 & $3^r,\ 2\cdot 3^{2r-1},\ 4\cdot 3^{2r-1},\ 5\cdot 3^{2r-1}$\\
\hline \rule[-2mm]{0mm}{6mm}
2 & 3 & 6 & $6,9,12$\\
\hline \rule[-2mm]{0mm}{6mm}
2 & 3 & 18 & $3^{2r+2}$\\
\hline \rule[-2mm]{0mm}{6mm}
2 & 3 & 27 & $27,54$\\
\hline \rule[-2mm]{0mm}{6mm}
3 & 5 & 15 & $15,30$\\
\hline
\end{tabular}
\end{table}


\begin{prop} \label{propreg}
Every quaternary triangular form $\T(a_1,a_2,a_3,a_4)$ in Table \ref{tablerk4} is regular.
\end{prop}

\begin{proof}
Let $\E=\T(a_1,a_2,a_3,a_4)$ ($a_1\le \cdots \le a_4$) be a triangular form in Table \ref{tablerk4}.
For simplicity, put $\mathbf{a}=(a_1,a_2,a_3,a_4)$ so that $\E=\T(\mathbf{a})$.
Let $n$ be any integer in $T^{\text{loc}}(\mathbf{a})$.
We shall show that $n\in T(\mathbf{a})$.
One may see that there are nonnegative integers $e_3,e_5$, and $e_7$ depending on $T(\mathbf{a})$ such that $\lambda_3^{e_3}\circ \lambda_5^{e_5}\circ \lambda_7^{e_7}(\T(\mathbf{a}))$ is universal.
We use induction on $s=e_3+e_5+e_7$.
If $s=0$, then $T(\mathbf{a})$ itself is universal and thus regular.
So we are done when $s=0$.
Now, assume that $\lambda_p(\T(\mathbf{a}))$ is regular, where $p\in P'$ and $\T(\mathbf{a})$ is $p$-unstable.
If $t(n,\mathbf{a})\equiv 0\Mod p$, then we have $n\in T(\mathbf{a})$ by Lemma \ref{lemreci}(ii).
From now on, we always assume that
$$
t(n,\mathbf{a})\not\equiv 0\Mod p.
$$
First, we deal with the case when $\T(\mathbf{a}[4])$ is regular.
Looking at Table \ref{tablerk4}, one can see that there is a unique odd prime, say $p$, for which $\mathcal L(\mathbf{a}[4])$ is $p$-unstable and $a_4$ is divisible by $p$.
One may use \cite[Theorem 1]{OM2} to show that
$$
\overline{DQ_p(\mathbf{a})}-\overline{DQ_p(\mathbf{a}[4])}\subset p\z_p.
$$
Then we have
$$
t(n,\mathbf{a}[4])=t(n,\mathbf{a})-a_4\in \overline{DQ_p(\mathbf{a})}\cap \z_p^{\times}\subset \overline{DQ_p(\mathbf{a}[4])}.
$$
Since $\overline{DQ_p(\mathbf{a}[4])}=\z_q$ for every odd prime $q\neq p$, it follows that $t(n,\mathbf{a}[4])$ is locally represented by $\mathcal L(\mathbf{a}[4])$.
Hence $n\in T^{\text{loc}}(\mathbf{a}[4])$.
From this and the regularity of $\T(\mathbf{a}[4])$, it follows that
$$
n\in T^{\text{loc}}(\mathbf{a}[4])=T(\mathbf{a}[4])\subset T(\mathbf{a}).
$$

Next, assume that $\T(\mathbf{a}[4])$ is irregular.
Then one may see that $\T(\mathbf{a}[4])$ is imprimitive regular for all such 18 forms in Table \ref{tablerk4}.
In this case, we only provide the proof for the form $\T(1,6,18,18)$ since all other cases can be dealt with in a similar manner.
Now, $\mathbf{a}=(1,6,18,18)$, $t(n,\mathbf{a})=8n+43$, and $p=3$.
From the assumptions that $n\in T^{\text{loc}}(\E)$ and $t(n,\mathbf{a})\not\equiv 0\Mod p$, we have
$$
t(n,\mathbf{a})\in \overline{DQ_3(\mathbf{a})}-3\z_3.
$$
This implies that $n\equiv 0\Mod 3$.
If we define
$$
\alpha=\begin{cases}0&\text{if}\ n\equiv 0\Mod{18},\\
3&\text{if}\ n\equiv 3\Mod{18},\\
6&\text{if}\ n\equiv 6\Mod{18},\\
45&\text{if}\ n\equiv 9\Mod{18},\\
66&\text{if}\ n\equiv 12\Mod{18},\\
15&\text{if}\ n\equiv 15\Mod{18},\end{cases}
$$
then one may see that $n-\alpha\equiv 0\Mod{18}$ and $8\alpha+1=x_1^2$ for some odd integer $x_1$.
On the other hand, one may easily show that any nonnegative integer divisible by 3 is represented by the regular ternary triangular form $\T(1,3,3)$.
Hence we have $\dfrac{n-\alpha}{6}\in T(1,3,3)$, which is equivalent to
$$
8\left(\frac{n-\alpha}{6}\right)+1+3+3=x_2^2+3x_3^2+3x_4^2
$$
for some odd integers $x_2,x_3$ and $x_4$.
From this follows that
\begin{align*}
8n+1+6+18+18&=8\alpha+1+8(n-\alpha)+6+18+18\\
&=x_1^2+6x_2^2+18x_3^2+18x_4^2.
\end{align*}
Since $x_i$ is odd for all $i=1,2,3,4$, we have $n\in T(1,6,18,18)$.
\end{proof}

Hereafter, we use the following notation.
For a set $S\subseteq \n \cup \{0\}$ and a vector $\mathbf{a}=(a_1,a_2,\dots,a_k)$, we define $\psi_S(\mathbf{a})$ to be the smallest positive integer in $S$ which is not represented by the triangular form $\T(\mathbf{a})$, i.e.,
$$
\psi_S(\mathbf{a})=\begin{cases}\min \{S-T(\mathbf{a})\}&\text{if}\ S-T(\mathbf{a})\neq \emptyset,\\
\infty&\text{otherwise}.\end{cases}
$$
We also define $\psi(\mathbf{a})=\psi_{T^{\text{loc}}(\mathbf{a})}(\mathbf{a})$.

\subsection{Drop structures on rank 4} \label{subsecdrop}

In this subsection, we find all drop structures on rank 4, which is equivalent to discover all pairs $(p,\mathbf{a})$, where $p$ is an odd prime and $\T(\mathbf{a})$ is a $p$-unstable new regular quaternary triangular form such that $\lambda_p(\T(\mathbf{a}))$ is old.
Throughout this subsection, we use the following notation.

Let $S_k$ denote the symmetric group of degree $k$.
Define
$$
\widetilde{\mathbf{a}}=\{(a_{\sigma(1)},a_{\sigma(2)},\dots,a_{\sigma(k)})\in \n^k : \sigma \in S_k\}.
$$
For a subset $S\subseteq \n^k$, we put
$$
\widetilde{S}=\bigcup_{\mathbf{a}\in S} \widetilde{\mathbf{a}}\subseteq \n^k.
$$
For $k=1,2,\dots$, let
$$
\mathcal{N}(k)=\{(a_1,a_2,\dots,a_k)\in \n^k : a_1\le a_2\le \cdots \le a_k\},
$$
and put $\mathcal{N}=\bigcup_{k=1}^{\infty}\mathcal{N}(k)$.
For any element $\mathbf{a}=(a_1,a_2,\dots,a_k)\in \n^k$, we denote by $\text{Sort}(\mathbf{a})$ the vector in $\mathcal{N}(k)$ obtained from $\mathbf{a}$ by sorting $a_1,a_2,\dots,a_k$ in ascending order.
For example, if $\mathbf{a}=(3,2,7,3,5)\in \n^5$, then $\text{Sort}(\mathbf{a})=(2,3,3,5,7)\in \mathcal{N}(5)$.
For a subset $X$ of $\n^k$, we put
$$
\text{Sort}(X)=\{ \text{Sort}(\mathbf{a}) : \mathbf{a}\in X\} \subseteq \mathcal{N}(k).
$$
Conversely, for a subset $Y$ of $\mathcal{N}(k)$, we let
$$
\text{Sort}^{-1}(Y)=\{ \mathbf{a}\in \n^k : \text{Sort}(\mathbf{a})\in Y\} \subseteq \n^k.
$$

Define
$$
X=\{nx : n\in \{1,3,5,9,25,27\},\ x\in \{1,2,\dots,8\} \} \subset \n.
$$
Let $Y_0$ be the set of all vectors $\mathbf{a}=(a_1,a_2,a_3)\in \mathcal{N}(3)$ such that the ternary triangular form $\T(\mathbf{a})$ is primitive regular and the quadratic $\q_p$-space $\q_p(\mathcal{L}(\mathbf{a}))$ is universal for every odd prime $p$.
One may use Theorem \ref{thmrk3} with Table \ref{tablerk3} to see that
$$
Y_0=
\left\{
\begin{array}{c}
(1,1,1),(1,1,2),(1,1,4),(1,1,5),(1,1,9),(1,1,18),\\[0.3em]
(1,2,2),(1,2,3),(1,2,4),(1,3,6),(1,3,18),(1,4,9),\\[0.3em]
(1,5,5),(1,5,25),(1,6,27),(1,9,9),(1,9,18),(2,3,9)
\end{array}
\right\}.
$$
We define
\begin{align*}
Y_1&=\{ p^{\delta}\mathbf{a} : p\in P',\ \delta \in \{1,2\},\ \mathbf{a}\in Y_0\}, \\
Y_2&=\bigcup_{p\in P',\ \mathbf{a}\in Y_0} \lambda_p^{-1}(\mathbf{a}),\\
Y_3&=\{ \mathbf{a}\in Y_2 : \psi(\mathbf{a})<\infty \},
\end{align*}
and
\begin{align*}
Z_1&=\{(a_1,a_2,a_3,a_4)\in \n^4 : a_1\in X,\ (a_2,a_3,a_4)\in Y_1,\ \gcd(a_1,a_2,a_3,a_4)=1\},\\
Z_2&=\{(a_1,a_2,a_3,a_4)\in \n^4 : a_1\in X,\ (a_2,a_3,a_4)\in Y_2\},\\
Z_3&=\left\{(a_1,a_2,a_3,a_4)\in \n^4 : \begin{array}{c}\mathbf{a}=(a_2,a_3,a_4)\in Y_3,\\ a_1\equiv 0\Mod {\xi(\mathbf{a})},\ a_1\le \psi(\mathbf{a})\end{array}\right\},
\end{align*}
where $\xi(\mathbf{a})$ is as defined in Remark \ref{rmkdivisor}.
Put
$$
Z_4=\sort(Z_1)\cup \sort(Z_2)\cup \sort(Z_3)\ (\subset \mathcal{N}(4)),
$$
and finally we let
$$
Z=\{ \mathbf{a}\in Z_4 : \psi(\mathbf{a})>124,\ \psi_{\n}(\mathbf{a})<\infty \}-W,
$$
where
$$
W=\sort(\{ \mathbf{a}=(a_1,a_2,a_3,a_4)\in \n^4 : \mathbf{a}[i]\in Y_0,\ a_i\equiv 0\Mod{\xi(\mathbf{a}[i])},\ 1\le i\le 4 \}).
$$


In the following consecutive four lemmas, we let $p$ be an odd prime and let $\mathbf{b}=(b_1,p^{e_2}b_2,p^{e_3}b_3,p^{e_4}b_4)\in \n^4$ with $e_2$ a nonnegative integer, $e_3,e_4$ positive integers, $b_1,b_2,b_3,b_4$ positive integers coprime to $p$ such that the quaternary triangular form $\T(\mathbf{b})$ is primitive $p$-unstable new regular.


\begin{lem} \label{lem11}
Under the notations given above, if $e_2>0$ and $\T(p^2b_1)$ is redundant in $\Lambda_p(\T(\mathbf{b}))$, then $b_1\in X_1$ and $\mathbf{b}[1]\in \widetilde{Y_1}$ so that $\mathbf{b}\in \widetilde{Z_1}$.
\end{lem}

\begin{proof}
Note that
$$
\lambda_p(\T(\mathbf{b}))=\T(p^{2-\delta}b_1,p^{e_2-\delta}b_2,p^{e_3-\delta}b_3,p^{e_4-\delta}b_4),
$$
where $\delta=\min(e_2,e_3,e_4,2)\in \{1,2\}$.
Put $\mathbf{c}=p^{-\delta}\mathbf{b}[1]=(p^{e_2-\delta}b_2,p^{e_3-\delta}b_3,p^{e_4-\delta}b_4)$.

First, we show that $\mathbf{b}[1]\in \widetilde{Y_1}$.
By Lemma \ref{lemnew}, $\mathcal L(\mathbf{c})$ is primitive regular with
$$
p^{2-\delta}b_1\z_q\subseteq \overline{DQ_q(\mathbf{c})}\ \ \text{for all}\ \ q\in P.
$$
Hence $\mathcal L(\mathbf{c})$ is a primitive regular ternary $\z$-lattice such that $\q_q\mathcal L(\mathbf{c})$ is universal for every odd prime $q$, i.e., $\mathbf{c}\in \widetilde{Y_0}$.
Thus we have $\mathbf{b}[1]\in \widetilde{Y_1}$.

Next, we show that $b_1\in X_1$.
If we define a set $A$ by
$$
A=\{q\in P-\{p\} : \overline{DQ_q(\mathbf{c})}\subsetneq \z_q\}.
$$
Since $\mathbf{c}\in \widetilde{Y_0}$, we have $A=\emptyset$, $\{3\}$ or $\{5\}$.
If $A=\emptyset$, then it follows that $\overline{DQ_q(\mathbf{b})}=\z_q$ for any odd prime $q\neq p$.
Assume that $A=\{q_0\}$, where $q_0=3$ or 5.
We define
$$
r_0=\begin{cases}3&\text{if}\ \ q_0=3,\\
2&\text{if}\ \ q_0=5.\end{cases}
$$
Note that
$$
b_1\z_{q_0}=p^2b_1\z_{q_0}\subseteq \overline{DQ_{q_0}(\mathbf{b}[1])}
$$
since $\T(p^2b_1)$ is redundant to $\T(\mathbf{b}[1])$.
By using Lemma \ref{lemlambdawell} with $q=q_0$, we may get a regular triangular form $\T(\mathbf{b}')=\T(b_1',p^{e_2}b_2',p^{e_3}b_3',p^{e_4}b_4')$ with $b_1'=q_0^{-r}b_1$ ($0\le r\le r_0$) such that
$$
\overline{DQ_q(\mathbf{b}')}\subseteq \overline{DQ_q(\mathbf{b}'[1])}=\z_q\ \ \text{for all}\ \ q\in P-\{p\}.
$$
If $p\in \{5,7\}$, then we have
\begin{align*}
&\left\vert \left\{ 0\le n\le \frac{p+3}{2} : n\not\in T(\mathbf{b}') \right\} \right\vert \\
=&\left\vert \left\{ 0\le n \le \frac{p+3}{2} : t(n,\mathbf{b}')\not\in \overline{DQ_p(\mathbf{b}')} \right\} \right\vert \\
\le&\left\vert \left\{ 0\le n \le p-1 : t(n,\mathbf{b}')\not\in \overline{DQ_p(\mathbf{b}')} \right\} \right\vert \\
\le&\frac{p+1}{2},
\end{align*}
and hence
$$
\left\vert \left\{ 0\le n\le \frac{p+3}{2} : n\in T(\mathbf{b}') \right\} \right\vert \ge 2
$$
so that $b_1'\le \dfrac{p+3}{2}\le 5$.
If $p=3$, then since
$$
t(3,\mathbf{b}')\in \overline{DQ_q(\mathbf{b}')},
$$
either $b_1'\le 3$ or $p^{e_i}b_i'=3$ for some $i\in \{2,3,4\}$.
We may assume the latter case holds since otherwise we are done.
Then by Lemma \ref{lemlocrep}(iii), we have
$$
\vert \{ 0\le n\le 8 : n\in T(\mathbf{b}') \} \vert \ge 4.
$$
From this and $T(3^{e_2}b_2',3^{e_3}b_3',3^{e_4}b_4')\subseteq 3\z$, it follows that $b_1'\le 8$.
Therefore we have $b_1\in X_1$.
This completes the proof.
\end{proof}


\begin{lem}
Under the notations given above, assume that $e_2=0$ and there is an $i\in \{1,2\}$ such that $\T(p^2b_i)$ is redundant in $\Lambda_p(\T(\mathbf{b}))$.
Then $b_i\in X_1$ and $\mathbf{b}[i]=(b_j,p^{e_3}b_3,p^{e_4}b_4)\in \widetilde{Y_2}$ so that $\mathbf{b}\in \widetilde{Z_2}$.
\end{lem}

\begin{proof}
The proof is quite similar to that of Lemma \ref{lem11} and omitted.
\end{proof}


\begin{lem} \label{lem12}
Under the notations given above, assume that $e_2>0$ and there is an $i\in \{2,3,4\}$ such that $\T(p^{e_i}b_i)$ is redundant in $\Lambda_p(\T(\mathbf{b}))$.
Then the ternary triangular form $\T(\mathbf{b}[i])$ is primitive and irregular, and we have $\mathbf{b}\in \widetilde{Z_3}$.
\end{lem}

\begin{proof}
For simplicity, put
$$
\mathbf{b}'=\Lambda_p(\mathbf{b})=(p^2b_1,p^{e_2}b_2,p^{e_3}b_3,p^{e_4}b_4).
$$
By Lemma \ref{lemnew}, the triangular form $\T(\lambda_p(\mathbf{b})[i])$ is primitive.
From this, one may easily deduce the primitivity of $\T(\mathbf{b}[i])$.

Suppose that $\T(\mathbf{b}[i])$ is regular.
Since $\T(p^{e_i}b_i)$ is redundant in $\T(\mathbf{b}')$, by Lemma \ref{lemnew}, we have
$$
p^{e_i}b_i\z_q\subseteq   \overline{DQ_q(\mathbf{b}'[i])} \subseteq \overline{DQ_q(\mathbf{b}[i])}\ \ \text{for all}\ \ q\in P.
$$
Since we have already shown that $\T(\mathbf{b}[i])$ is primitive, Lemma \ref{lemnew} implies that $\T(\mathbf{b})$ is old, which is a contradiction.
Thus we showed that $\T(\mathbf{b}[i])$ is primitive and irregular.
It follows that $\mathbf{b}[i]\in \widetilde{Y_3}$.
Since $\T(p^{e_i}b_i)$ is redundant in $\T(\mathbf{b}')$, $p^{e_i}b_i$ must be divisible by $\xi(\mathbf{b}'[i])$.
One may easily see that $\xi(\mathbf{b}'[i])$ is divisible by $\xi(\mathbf{b}[i])$ (and in fact, they coincide in this case).
This implies that $p^{e_i}b_i\equiv 0\Mod {\xi(\mathbf{b}[i])}$.
If $p^{e_i}b_i>\psi(\mathbf{b}[i])$, then one may easily deduce that
$$
\psi(\mathbf{b}[i])\in T^{\text{loc}}(\mathbf{b})-T(\mathbf{b}),
$$
which is absurd.
Hence we have $p^{e_i}b_i\le \psi(\mathbf{b}[i])$.
This completes the proof.
\end{proof}


\begin{lem} \label{lem22}
Under the notations given above, assume that $e_2=0$ and there is an $i\in \{3,4\}$ such that $\T(p^{e_i}b_i)$ is redundant in $\Lambda_p(\T(\mathbf{b}))$.
Then the ternary triangular form $\T(\mathbf{b}[i])$ is primitive and irregular, and we have $\mathbf{b}\in \widetilde{Z_3}$.
\end{lem}

\begin{proof}
The proof is quite similar to that of Lemma \ref{lem12} and left to the reader.
\end{proof}


\begin{prop} \label{propdrop}
There are exactly 27 drop structures on rank 4 as listed in Table \ref{tabledrop}.
\end{prop}

\begin{proof}
Let $p$ be an odd prime and $\T(\mathbf{a})$ be a $p$-unstable new regular quaternary triangular form such that $\lambda_p(\T(\mathbf{a}))$ is old.
By Lemmas \ref{lem11}-\ref{lem22}, we have $\mathbf{a}\in Z_4$.
Note that $\T(\mathbf{b})$ is old regular for any $\mathbf{b}\in W$.
Now one may easily deduce $\mathbf{a}\in Z$ from the assumption that $\E$ is $p$-unstable new regular.
With the help of computer, one may check that $\vert Z\vert=78$.
Let $\mathbf{c}$ be a vector in $Z$ which does not appear in Table \ref{tabledrop}.
Then one may directly check that $\T(\mathbf{c})$ is a new regular quaternary triangular form that is $q$-unstable for exactly one odd prime $q$ such that $\lambda_q(\T(\mathbf{c}))$ is also new so that there are positive integers $i$ and $\nu$ with $5\le i\le 31$ such that $\lambda_q^{\nu}(\T(\mathbf{c}))=\E_i$.
This completes the proof.
\end{proof}


\begin{table}[ht]
\caption{The drop structures on rank 4} 
\label{tabledrop}
\begin{tabular}{lcll}
\hline
\multicolumn{1}{c}{top} & $p$ & \multicolumn{1}{c}{$\lambda_p(\E_i)$} & \multicolumn{1}{c}{bottom}\\
\hline
$\E_5=\T(1,1,3,18)$ & 3 & $\T(1,3,3,6)$ & $\T(1,3,6)$\\
\hline
$\E_6=\T(1,2,5,25)$ & 5 & $\T(1,5,5,10)$ & $\T(1,5,5)$\\
\hline
$\E_7=\T(1,3,3,3)$ & 3 & $\T(1,1,1,3)$ & $\T(1,1,1)$\\
\hline
$\E_8=\T(1,3,3,12)$ & 3 & $\T(1,1,3,4)$ & $\T(1,1,4)$\\
\hline
$\E_9=\T(1,3,3,15)$ & 3 & $\T(1,1,3,5)$ & $\T(1,1,5)$\\
\hline
$\E_{10}=\T(1,3,4,18)$ & 3 & $\T(1,3,6,12)$ & $\T(1,3,6)$\\
\hline
$\E_{11}=\T(1,3,27,54)$ & 3 & $\T(1,3,9,18)$ & $\T(1,3,18)$\\
\hline
$\E_{12}=\T(1,5,10,10)$ & 5 & $\T(1,2,2,5)$ & $\T(1,2,2)$\\
\hline
$\E_{13}=\T(1,5,10,15)$ & 5 & $\T(1,2,3,5)$ & $\T(1,2,3)$\\
\hline
$\E_{14}=\T(1,5,10,20)$ & 5 & $\T(1,2,4,5)$ & $\T(1,2,4)$\\
\hline
$\E_{15}=\T(1,7,7,7)$ & 7 & $\T(1,1,1,7)$ & $\T(1,1,1)$\\
\hline
$\E_{16}=\T(1,7,7,14)$ & 7 & $\T(1,1,2,7)$ & $\T(1,1,2)$\\
\hline
$\E_{17}=\T(1,7,7,28)$ & 7 & $\T(1,1,4,7)$ & $\T(1,1,4)$\\
\hline
$\E_{18}=\T(1,7,7,35)$ & 7 & $\T(1,1,5,7)$ & $\T(1,1,5)$\\
\hline
$\E_{19}=\T(1,7,14,14)$ & 7 & $\T(1,2,2,7)$ & $\T(1,2,2)$\\
\hline
$\E_{20}=\T(1,7,14,21)$ & 7 & $\T(1,2,3,7)$ & $\T(1,2,3)$\\
\hline
$\E_{21}=\T(1,7,14,28)$ & 7 & $\T(1,2,4,7)$ & $\T(1,2,4)$\\
\hline
$\E_{22}=\T(2,2,3,9)$ & 3 & $\T(1,3,6,6)$ & $\T(1,3,6)$\\
\hline
$\E_{23}=\T(2,3,3,3)$ & 3 & $\T(1,1,1,6)$ & $\T(1,1,1)$\\
\hline
$\E_{24}=\T(2,3,3,6)$ & 3 & $\T(1,1,2,6)$ & $\T(1,1,2)$\\
\hline
$\E_{25}=\T(2,3,3,12)$ & 3 & $\T(1,1,4,6)$ & $\T(1,1,4)$\\
\hline
$\E_{26}=\T(2,3,3,15)$ & 3 & $\T(1,1,5,6)$ & $\T(1,1,5)$\\
\hline
$\E_{27}=\T(2,3,6,6)$ & 3 & $\T(1,2,2,6)$ & $\T(1,2,2)$\\
\hline
$\E_{28}=\T(2,3,6,9)$ & 3 & $\T(1,2,3,6)$ & $\T(1,2,3)$\\
\hline
$\E_{29}=\T(2,3,6,12)$ & 3 & $\T(1,2,4,6)$ & $\T(1,2,4)$\\
\hline
${}^{\dag}\E_{30}=\T(3,5,15,15)$ & 3 & $\T(1,5,5,15)$ & $\T(1,5,5)$\\
\hline
${}^{\dag}\E_{31}=\T(3,5,15,30)$ & 5 & $\T(1,3,6,15)$ & $\T(1,3,6)$\\
\hline
\end{tabular}
\end{table}


\begin{rmk} \label{rmktwoways}
In Table \ref{tabledrop}, for $5\le i\le 29$, the triangular form $\E_i$ is $q$-stable for all odd primes $q$ except the prime $p$ in the same line.
However, both $\E_{30}=\T(3,5,15,15)$ and $\E_{31}=\T(3,5,15,30)$ are $q$-unstable for $q=3,5$.
Though two triangular forms $\lambda_3(\E_{30})$ and $\lambda_5(\E_{31})$ are old regular, if we take the road not taken, we see that
\begin{align*}
\begin{split}
&\lambda_5(\E_{30})=\T(1,3,3,15)=\E_9,\\ 
&\lambda_3(\E_{31})=\T(1,5,10,15)=\E_{13},
\end{split}
\end{align*}
where both forms $\E_9$ and $\E_{13}$ are new regular.
\end{rmk}


\begin{rmk}
As explained in Subsection \ref{subsecdrop}, via $\lambda_p$-transformations for some primes $p\in P'$, any new regular quaternary triangular form $\E$ either flows into the top of a drop structure or arrive at one of the stable new regular quaternary triangular forms.
By this and Remark \ref{rmktwoways}, there are nonnegative integers $e_3,e_5$ and $e_7$ such that
$$
\lambda_3^{e_3}\circ \lambda_5^{e_5}\circ \lambda_7^{e_7}(\E)=\E_i
$$
for some $1\le i\le 29$.
\end{rmk}

\subsection{Preimage under the Watson transformation}

In this subsection, we describe how the stream above the new regular quaternary triangular form $\E_i$ looks like, for $i=1,2,\dots,29$.
Since the procedure is quite similar to each other, we look into the case of $\E_{24}=\T(2,3,3,6)$ that is contained in the river with mouth $\T(1,1,2)$ described in Table \ref{river112}.
For any regular triangular form $\E$, we define
$$
s(\E)=\bigcup_{p\in P'}\{ \mathcal{F}\in \lambda_p^{-1}(\E) : \mathcal{F}\ \text{is}\ p\text{-unstable new regular}\}.
$$
Note that, for a given new regular quaternary triangular form $\E$, one may easily determine $s(\E)$ with the help of computer and Theorem \ref{thmrk4}.
For simplicity, we define for $r=1,2,\dots$,
\begin{align*}
\mathcal{F}_{1,r}&=\T(2,3,3,2\cdot 3^{2r-1}),\\
\mathcal{F}_{2,r}&=\T(1,1,6,2\cdot 3^{2r}),\\
\mathcal{F}_{3,r}&=\T(1,6,9,2\cdot 3^{2r}).
\end{align*}
Let $r$ be any fixed positive integer.
For every $p\in P'$, we compute the set $\lambda_p^{-1}(\mathcal{F}_{1,r})$ and the value $\psi(a_1,a_2,a_3,a_4)$ for every element $\T(a_1,a_2,a_3,a_4)$ in $\lambda_p^{-1}(\mathcal{F}_{1,r})$.
These data are provided in Table \ref{table2236}, and we note that the two forms (marked with $\dag$) $\T(1,2,6,9)$ and $\T(2,6,9,9)$ are contained only in $\lambda_p^{-1}(\mathcal{F}_{1,1})$ and are not contained in $\lambda_p^{-1}(\mathcal{F}_{1,r})$ for $r\ge 2$, so that
$$
\vert \lambda_p^{-1}(\mathcal{F}_{1,r})\vert =\begin{cases} 3&\text{if}\ p=3\ \text{and}\ r\ge 2,\\
5&\text{if}\ p\in \{5,7\}, \ \text{or}\ p=3\ \text{and}\ r=1.\end{cases}
$$
Among the elements of $\lambda_p^{-1}(\mathcal{F}_{1,r})$, we shall find $p$-unstable new regular triangular forms.

For any positive integer $r$, one may easily verify the following:
\begin{align*}
s(\mathcal{F}_{1,r})&=\{ \mathcal{F}_{2,r},\mathcal{F}_{3,r}\},\\
s(\mathcal{F}_{2,r})&=\{ \mathcal{F}_{1,r+1}\},\\
s(\mathcal{F}_{3,r})&=\begin{cases} \{ \T(2,3,27,54)\}&\text{if}\ r=1,\\
\emptyset&\text{if}\ r\ge 2.\end{cases}
\end{align*}
If we let $\mathcal{F}_{4,r}=\T(2,3,27,2\cdot 3^{2r+1})$, then obviously $\mathcal{F}_{4,r}\in \lambda_3^{-1}(\mathcal{F}_{3,r})$ for every $r\in \n$ but the quaternary triangular form $\mathcal{F}_{4,r}$ cannot be regular for any $r\ge 2$ because the ternary triangular form $\T(2,3,27)$ is irregular with $\psi(2,3,27)=54$ while $2\cdot 3^{2r+1}>54$ when $r>1$.


\begin{table}[ht]
\caption{All elements $\T(a_1,a_2,a_3,a_4)$ in $\lambda_p^{-1}(\T(2,3,3,2\cdot 3^{2r-1}))$, $p\in P'$}
\begin{tabular}{ccc}
\hline
$p$ & $(a_1,a_2,a_3,a_4)$ & $\psi(a_1,a_2,a_3,a_4)$ \\
\hline
\multirow{5}{*}{3} & $(1,1,6,2\cdot 3^{2r})$ & $\infty$ \\
& $(1,6,9,2\cdot 3^{2r})$ & $\infty$ \\
& $(2,27,27,2\cdot 3^{2r+1})$ & 3 \\
& ${}^{\dag}(1,2,6,9)$ & 4 \\
& ${}^{\dag}(2,6,9,9)$ & 3 \\
\hline
\multirow{5}{*}{5} & $(2,75,75,2\cdot 3^{2r-1})$ & 3 \\
& $(2,75,75,2\cdot 3^{2r-1}\cdot 5^2)$ & 5 \\
& $(3,50,75,2\cdot 3^{2r-1})$ & 2 \\
& $(3,50,75,2\cdot 3^{2r-1}\cdot 5^2)$ & 2 \\
& $(50,75,75,2\cdot 3^{2r-1})$ & 5 \\
\hline
\multirow{5}{*}{7} & $(2,147,147,2\cdot 3^{2r-1})$ & 3 \\
& $(2,147,147,2\cdot 3^{2r-1}\cdot 7^2)$ & 3 \\
& $(3,3,98,2\cdot 3^{2r-1}\cdot 7^2)$ & 2 \\
& $(3,98,147,2\cdot 3^{2r-1})$ & 2 \\
& $(3,98,147,2\cdot 3^{2r-1}\cdot 7^2)$ & 2 \\
\hline
\end{tabular}
\label{table2236}
\end{table}


In Table \ref{tablestream}, the streams upon $\E_i$ for all $i=1,2,\dots,29$ are given.
In the table, MS denotes the mainstreams upon $\E_i$, PT denotes the periodic tributaries and $S$ denotes the sporadic forms (for the terms, see Subsection \ref{subsecriv}).
Let us explain the contents given in the table.
For each $i=1,2,\dots,29$, we give $i$, $\E_i=\T(a_1,a_2,a_3,a_4)$, and the vector $(\# \text{MS},\# \text{PT},\# \text{S})$ in the first line and we provide MS, PT, S if exist.
Let $i=3$. Upon $\E_3=\T(1,1,3,7)$, there is a unique mainstream, a periodic tributary and a sporadic form upon $\E_3$.
The mainstream consists of two infinite families $\T(1,1,3,3^{2r-2}7)$ ($r\in \n$) and $\T(1,3,3,3^{2r-1}7)$ ($r\in \n$) which take turns appearing.
Obviously, $\lambda_3(\T(1,1,3,3^{2r}7)=(\T(1,3,3,3^{2r-1}7)$ and $\lambda_3(\T(1,3,3,3^{2r-1}7)=\T(1,1,3,3^{2r-2}7)$.
For every $r\in \n$, there is a tributary 
$$
(1,3,27,3^{2r+1}7)\ral{3}(1,3,9,3^{2r}7)\ral{3}(1,3,3,3^{2r-1}7)
$$
that feeds into the mainstream.
On the other hand, only when $r=1$, there is a sporadic form $\T(1,7,7,3^{2r-1}7)$ flows into $\T(1,1,3,3^{2r-2}7)$ via $\lambda_7$.
This sporadic form have failed to grow to a periodic tributary mainly due to the irregularity of its ternary section $\T(1,7,7)$, that is to say, the quaternary triangular form $\T(1,7,7,a_4)$ cannot be regular for $a_4>\psi(1,7,7)=41$.


\begin{table}[ht]
\caption{Stream upon $\E_i=\T(a_1,a_2,a_3,a_4)$, $i=1,2,\dots,29$}
\begin{tabular}{c p{9.1cm} c}
\hline
$i=1$ & $(a_1,a_2,a_3,a_4)=(1,1,3,3)$ & $(1,1,0)$ \\
\hline
MS & $(1,1,3,3^{2r-1})\lral{3}(1,3,3,3^{2r})$ \\
\hline
PT & $(1,3,27,3^{2r+2})\ral{3}(1,3,9,3^{2r+1})\ral{3}(1,3,3,3^{2r})$ \\
\hline
\hline
2 & $(1,1,3,6)$ & $(2,2,0)$ \\
\hline
\multirow{2}{*}{MS} & $(1,1,3,2\cdot 3^{2r-1})\lral{3}(1,3,3,2\cdot 3^{2r})$ \\
& $(1,1,6,3^{2r-1})\lral{3}(2,3,3,3^{2r})$ \\
\hline
\multirow{2}{*}{PT} & $(1,3,27,2\cdot 3^{2r+2})\ral{3}(1,3,9,2\cdot 3^{2r+1})\ral{3}(1,3,3,2\cdot 3^{2r})$ \\
& $(1,6,9,3^{2r+1})\ral{3}(2,3,3,3^{2r})$ \\
\hline
\hline
3 & $(1,1,3,7)$ & $(1,1,1)$ \\
\hline
MS & $(1,1,3,3^{2r-2}\cdot 7)\lral{3}(1,3,3,3^{2r-1}\cdot 7)$ \\
\hline
PT & $(1,3,27,3^{2r+1}\cdot 7)\ral{3}(1,3,9,3^{2r}\cdot 7)\ral{3}(1,3,3,3^{2r-1}\cdot 7)$ \\
\hline
S & $(1,7,7,21)\ral{7}(1,1,3,7)$ \\
\hline
\hline
4 & $(1,1,3,8)$ & $(1,1,0)$ \\
\hline
MS & $(1,1,3,2^3\cdot 3^{2r-2})\lral{3}(1,3,3,2^3\cdot 3^{2r-1})$ \\
\hline
PT & $(1,3,27,2^3\cdot 3^{2r+1})\ral{3}(1,3,9,2^3\cdot 3^{2r})\ral{3}(1,3,3,2^3\cdot 3^{2r-1})$ \\
\hline
\hline
5 & $(1,1,3,18)$ & $(1,1,0)$ \\
\hline
MS & $(1,1,3,2\cdot 3^{2r})\lral{3}(1,3,3,2\cdot 3^{2r+1})$ \\
\hline
PT & $(1,3,27,2\cdot 3^{2r+3})\ral{3}(1,3,9,2\cdot 3^{2r+2})\ral{3}(1,3,3,2\cdot 3^{2r+1})$ \\
\hline
\hline
6 & $(1,2,5,25)$ & $(1,0,0)$ \\
\hline
MS & $(1,2,5,5^{2r})\lral{5}(1,5,10,5^{2r+1})$ \\
\hline
\hline
7 & $(1,3,3,3)$ & $(1,1,0)$ \\
\hline
MS & $(1,3,3,3^{2r-1})\lral{3}(1,1,3,3^{2r})$ \\
\hline
PT & $(1,3,27,3^{2r+1})\ral{3}(1,3,9,3^{2r})\ral{3}(1,3,3,3^{2r-1})$ \\
\hline
\hline
8 & $(1,3,3,12)$ & $(2,1,0)$ \\
\hline
\multirow{2}{*}{MS} & $(1,3,12,3^{2r-1})\lral{3}(1,3,4,3^{2r})$ & \\
& $(1,3,3,2^2\cdot 3^{2r-1})\lral{3}(1,1,3,2^2\cdot 3^{2r})$ \\
\hline
PT & $(1,3,27,2^2\cdot 3^{2r+1})\ral{3}(1,3,9,2^2\cdot 3^{2r})\ral{3}(1,3,3,2^2\cdot 3^{2r-1})$ \\
\hline
\hline
9 & $(1,3,3,15)$ & $(1,1,1)$ \\
\hline
MS & $(1,3,3,3^{2r-1}\cdot 5)\lral{3}(1,1,3,3^{2r}\cdot 5)$ \\
\hline
PT & $(1,3,27,3^{2r+1}\cdot 5)\ral{3}(1,3,9,3^{2r}\cdot 5)\ral{3}(1,3,3,3^{2r-1}\cdot 5)$ \\
\hline
S & $(3,5,15,15)\ral{5}(1,3,3,15)$ \\
\hline
\hline
10 & $(1,3,4,18)$ & $(1,0,0)$ \\
\hline
MS & $(1,3,4,2\cdot 3^{2r})\lral{3}(1,3,12,2\cdot 3^{2r+1})$ \\
\hline
\hline
11 & $(1,3,27,54)$ & $(0,0,0)$ \\
\hline
\hline
12 & $(1,5,10,10)$ & $(1,0,0)$ \\
\hline
MS & $(1,5,10,2\cdot 5^{2r-1})\lral{5}(1,2,5,2\cdot 5^{2r})$ \\
\hline
\end{tabular}
\label{tablestream}
\end{table}


\begin{table}[ht]
\ContinuedFloat
\caption{Continued}
\begin{tabular}{c p{8.9cm} c}
\hline
$i=13$ & $(a_1,a_2,a_3,a_4)=(1,5,10,15)$ & $(1,0,1)$ \\
\hline
MS & $(1,5,10,3\cdot 5^{2r-1})\lral{5}(1,2,5,3\cdot 5^{2r})$ \\
\hline
S & $(3,5,15,30)\ral{3}(1,5,10,15)$ \\
\hline
\hline
14 & $(1,5,10,20)$ & $(1,0,0)$ \\
\hline
MS & $(1,5,10,2^2\cdot 5^{2r-1})\lral{5}(1,2,5,2^2\cdot 5^{2r})$ \\
\hline
\hline
15 & $(1,7,7,7)$ & $(0,0,0)$ \\
\hline
\hline
16 & $(1,7,7,14)$ & $(0,0,0)$ \\
\hline
\hline
17 & $(1,7,7,28)$ & $(0,0,0)$ \\
\hline
\hline
18 & $(1,7,7,35)$ & $(0,0,0)$ \\
\hline
\hline
19 & $(1,7,14,14)$ & $(0,0,0)$ \\
\hline
\hline
20 & $(1,7,14,21)$ & $(0,0,0)$ \\
\hline
\hline
21 & $(1,7,14,28)$ & $(0,0,0)$ \\
\hline
\hline
22 & $(2,2,3,9)$ & $(1,1,0)$ \\
\hline
MS & $(2,2,3,3^{2r})\lral{3}(1,6,6,3^{2r+1})$ \\
\hline
PT & $(2,3,18,3^{2r+2})\ral{3}(1,6,6,3^{2r+1})$ \\
\hline
\hline
23 & $(2,3,3,3)$ & $(1,1,1)$ \\
\hline
MS & $(2,3,3,3^{2r-1})\lral{3}(1,1,6,3^{2r})$ \\
\hline
PT & $(1,6,9,3^{2r})\ral{3}(2,3,3,3^{2r-1})$ \\
\hline
S & $(2,3,27,27)\ral{3}(1,6,9,9)$ \\
\hline
\hline
24 & $(2,3,3,6)$ & $(1,1,1)$ \\
\hline
MS & $(2,3,3,2\cdot 3^{2r-1})\lral{3}(1,1,6,2\cdot 3^{2r})$ \\
\hline
PT & $(1,6,9,2\cdot 3^{2r})\ral{3}(2,3,3,2\cdot 3^{2r-1})$ \\
\hline
S & $(2,3,27,54)\ral{3}(1,6,9,18)$ \\
\hline
\hline
25 & $(2,3,3,12)$ & $(1,1,0)$ \\
\hline
MS & $(2,3,3,2^2\cdot 3^{2r-1})\lral{3}(1,1,6,2^2\cdot 3^{2r})$ \\
\hline
PT & $(1,6,9,2^2\cdot 3^{2r})\ral{3}(2,3,3,2^2\cdot 3^{2r-1})$ \\
\hline
\hline
26 & $(2,3,3,15)$ & $(1,1,0)$ \\
\hline
MS & $(2,3,3,3^{2r-1}\cdot 5)\lral{3}(1,1,6,3^{2r}\cdot 5)$ \\
\hline
PT & $(1,6,9,3^{2r}\cdot 5)\ral{3}(2,3,3,3^{2r-1}\cdot 5)$ \\
\hline
\hline
27 & $(2,3,6,6)$ & $(0,0,1)$ \\
\hline
S & $(1,6,18,18)\ral{3}(2,3,6,6)$ \\
\hline
\hline
28 & $(2,3,6,9)$ & $(0,0,1)$ \\
\hline
S & $(1,6,18,27)\ral{3}(2,3,6,9)$ \\
\hline
\hline
29 & $(2,3,6,12)$ & $(0,0,1)$ \\
\hline
S & $(1,6,18,36)\ral{3}(2,3,6,12)$ \\
\hline
\end{tabular}
\end{table}

\section{Regular triangular forms of rank exceeding 2} \label{sechigh}

We introduce some notation that will be used throughout this section.
Define a set $A$ by
$$
A=\{ \pm 1\} \times \{ \pm 1\} \times \{ \pm 1\} \times \{ n\in \z : 0\le n\le 104\}.
$$
We fix an element $(\eta_3,\eta_5,\eta_7,\bar{\alpha})$ in $A$ for a moment.
Note that every object introduced in the next paragraph depends on the choice of this element.

Let
$$
S=\left\{ n\in \n : \left( \frac{8n+\bar{\alpha}}{p}\right)=\eta_p\ \text{for all}\ \ p\in P'\right\}.
$$
We denote the $i$-th smallest element in $S$ by $s_i$.
For example, if $(\eta_3,\eta_5,\eta_7,\bar{\alpha})=(1,1,1,0)$, then $s_1=2$, $s_2=8$, $s_3=23$, etc.
Define
\begin{align*}
U_1&=\{ b_1\in \n : b_1\le s_1\},\\
U_2&=\{ (b_1,b_2)\in \n^2 : b_1\in U_1,\ b_1\le b_2\le \psi_{S}(b_1)\}.
\end{align*}
For $q\in P'$, we define $\delta_q$ to be the smallest positive integer satisfying $\left(\dfrac{\delta_q}{q}\right)=\eta_q$.
For each vector $(b_1,b_2)\in U_2$ and a prime $q\in P'$, we define a set $T(b_1,b_2,q)$ to be
$$
\begin{cases}\{ n\in \n : 8n+b_1+b_2+\delta_q\ra \langle b_1,b_2,\delta_q\rangle \ \text{over}\ \z_q\}&\text{if}\ \left(\dfrac{b_i}{q}\right)\neq \eta_q\ \text{for}\ i=1,2,\\
\{ n\in \n : 8n+b_1+b_2\ra \langle b_1,b_2\rangle \ \text{over}\ \z_q\}&\text{otherwise},\end{cases}
$$
and put $T(b_1,b_2)=\bigcap_{q\in P'}T(b_1,b_2,q)$.
Define
$$
U_3=\{ (b_1,b_2,b_3)\in \n^3 : (b_1,b_2)\in U_2,\ b_2\le b_3\le \psi_{T(b_1,b_2)}(b_1,b_2)\}.
$$
We define a subset $U'_3$ of $U_3$ by
$$
U'_3=\{ (b_1,b_2,b_3)\in U_3 : b_3\le \psi(b_1,b_2,b_3)\}.
$$

Now we put
$$
U=\bigcup_{(\eta_3,\eta_5,\eta_7,\bar{\alpha})\in A}U'_3(\eta_3,\eta_5,\eta_7,\bar{\alpha}).
$$
With the help of computer, one may easily compute the set $U$.
Note that $\vert U\vert=77$ and all the elements of $U$ are listed in Table \ref{tableu3}.


\begin{table}[ht]
\caption{Elements $(a,b,c)$ of $U$ and their type}
\begin{small}
\begin{tabular}{cc|cc|cc|cc|cc}
\hline
$(a,b,c)$ & Type & \multicolumn{6}{c}{} \\
\hline
$(1,1,1)$ & $\mathbf{3'}$ & $(1,1,2)$ & $\mathbf{3'}$ & $(1,1,3)$ & $\mathbf{1}$ & $(1,1,4)$ & $\mathbf{3'}$ & $(1,1,5)$ & $ \mathbf{3'}$ \\
\hline
$(1,1,6)$ & $\mathbf{1}$ & $(1,1,9)$ & $ \mathbf{0'}$ & $(1,1,12)$ & $ \mathbf{1'}$ & $(1,1,18)$ & $\mathbf{0'}$ & $(1,1,21)$ & $\mathbf{2'}$ \\
\hline
$(1,2,2)$ & $\mathbf{3'}$ & $(1,2,3)$ & $\mathbf{3'}$ & $(1,2,4)$ & $\mathbf{3'}$ & $(1,2,5)$ & $\mathbf{1}$ & $(1,2,7)$ & $\mathbf{6'}$ \\
\hline
$(1,2,10)$ & $\mathbf{1'}$ & $(1,3,3)$ & $\mathbf{1}$ & $(1,3,4)$ & $\mathbf{1}$ & $(1,3,6)$ & $\mathbf{0'}$ & $(1,3,9)$ & $\mathbf{1}$ \\
\hline
$(1,3,10)$ & $\mathbf{2'}$ & $(1,3,12)$ & $\mathbf{1}$ & $(1,3,18)$ & $\mathbf{0'}$ & $(1,3,27)$ & $\mathbf{1}$ & $(1,3,30)$ & $\mathbf{2'}$ \\
\hline
$(1,4,6)$ & $\mathbf{1'}$ & $(1,4,9)$ & $\mathbf{0'}$ & $(1,5,5)$ & $\mathbf{0'}$ & $(1,5,10)$ & $\mathbf{1}$ & $(1,5,25)$ & $\mathbf{0'}$ \\
\hline
$(1,6,6)$ & $\mathbf{1}$ & $(1,6,9)$ & $\mathbf{1}$ & $(1,6,18)$ & $\mathbf{6}$ & $(1,6,27)$ & $\mathbf{0'}$ & $(1,7,7)$ & $\mathbf{6}$ \\
\hline
$(1,7,14)$ & $\mathbf{6}$ & $(1,9,9)$ & $\mathbf{0'}$ & $(1,9,12)$ & $\mathbf{1'}$ & $(1,9,18)$ & $\mathbf{0'}$ & $(1,9,21)$ & $\mathbf{2'}$ \\
\hline
$(1,5,45)$ & $\mathbf{6'}$ & $(1,21,21)$ & $\mathbf{2'}$ & $(2,2,3)$ & $\mathbf{1}$ & $(2,3,3)$ & $\mathbf{1}$ & $(2,3,6)$ & $\mathbf{6}$ \\
\hline
$(2,3,9)$ & $\mathbf{0'}$ & $(2,3,12)$ & $\mathbf{1'}$ & $(2,3,18)$ & $\mathbf{1}$ & $(2,3,27)$ & $\mathbf{6}$ & $(2,5,10)$ & $\mathbf{1'}$ \\
\hline
$(2,5,15)$ & $\mathbf{2'}$ & $(2,15,45)$ & $\mathbf{6'}$ & $(3,3,3)$ & $\mathbf{4'}$ & $(3,3,4)$ & $\mathbf{1'}$ & $(3,3,6)$ & $\mathbf{4'}$ \\
\hline
$(3,3,7)$ & $\mathbf{2'}$ & $(3,3,9)$ & $\mathbf{4'}$ & $(3,3,12)$ & $\mathbf{4'}$ & $(3,3,15)$ & $\mathbf{4'}$ & $(3,5,9)$ & $\mathbf{6'}$ \\
\hline
$(3,5,15)$ & $\mathbf{6}$ & $(3,6,6)$ & $\mathbf{4'}$ & $(3,6,9)$ & $\mathbf{4'}$ & $(3,6,12)$ & $\mathbf{4'}$ & $(3,6,15)$ & $\mathbf{5'}$ \\
\hline
$(3,6,21)$ & $\mathbf{6'}$ & $(3,6,30)$ & $\mathbf{5'}$ & $(3,7,7)$ & $\mathbf{2'}$ & $(3,7,63)$ & $\mathbf{2'}$ & $(3,9,10)$ & $\mathbf{6'}$ \\
\hline
$(3,15,15)$ & $\mathbf{5'}$ & $(3,15,30)$ & $\mathbf{5'}$ & $(3,15,75)$ & $\mathbf{5'}$ & $(3,21,21)$ & $\mathbf{6'}$ & $(3,21,42)$ & $\mathbf{6'}$ \\
\hline
$(5,6,15)$ & $\mathbf{2'}$ & $(6,15,30)$ & $\mathbf{5'}$ \\
\hline
\end{tabular}
\end{small}
\label{tableu3}
\end{table}


\begin{lem} \label{lemter}
Let $\T(a_1,a_2,\dots,a_k)$ be a regular $k$-ary triangular form with $k\ge 4$ and $a_1\le a_2\le \cdots \le a_k$.
Then the vector $(a_1,a_2,a_3)$ appears in Table \ref{tableu3}.
\end{lem}

\begin{proof}
Let $\E=\T(a_1,a_2,\dots,a_k)$ be a new regular triangular form with $k\ge 4$ and let $L=\langle a_1,a_2,\dots,a_k\rangle$.
Put $\alpha=a_1+a_2+\cdots+a_k$ and let $\bar{\alpha}$ denote the smallest nonnegative integer congruent to $\alpha$ modulo $105$.
For each prime $q\in P'$, we define $\eta_q=\left(\dfrac{a_{i_q}}{q}\right)$, where $i_q$ is the smallest index such that $a_{i_q}$ is coprime to $q$.
Note that any integer $u$ satisfying $\left( \dfrac{u}{q}\right)=\eta_q$ is represented by $L_q$.
Hence any integer in the set $S(\eta_3,\eta_5,\eta_7,\bar{\alpha})$ is locally represented by $\T(a_1,a_2,\dots,a_k)$.
Hence we have $a_1\le s_1(\eta_3,\eta_5,\eta_7,\bar{\alpha})$, and $a_2\le \psi_{S(\eta_3,\eta_5,\eta_7,\bar{\alpha})}(a_1)$.
This shows that $(a_1,a_2)\in U_2$.
For a positive integer $n$ and a prime $q\in P'$, if $8n+a_1+a_2$ is represented by $\langle a_1,a_2\rangle$ over $\z_q$, then $8n+\alpha$ is represented by $L_q$.
If $\left( \dfrac{a_i}{p}\right)\neq \eta_q$ for some $q\in P'$, then we have $i_q\ge 3$.
Thus if a positive integer $n$ satisfies that $8n+a_1+a_2+a_{i_q}$ is represented by $\langle a_1,a_2,a_{i_q}\rangle$, then $8n+\alpha$ is represented by $L_q$.
One may easily deduce from this that $(a_1,a_2,a_3)\in U_3$.
Since $\T(a_1,a_2,a_3,\dots,a_k)$ is regular with $a_1\le a_2\le a_3\le \cdots \le a_k$, we must have $a_3\le \psi(a_1,a_2,a_3)$.
Therefore, $(a_1,a_2,a_3)\in U'_3$.
This completes the proof.
\end{proof}


\begin{proof}[Proof of Theorem \ref{thmrk5}]
Suppose that $\E=\T(a_1,a_2,\dots,a_k)$ is a new regular $k$-ary triangular form with $a_1\le a_2\le \cdots \le a_k$ such that $k\ge 5$.
By Lemma \ref{lemter}, the vector $(a_1,a_2,a_3)$ should appear in Table \ref{tableu3}.
We here only provide the proof for the representative case for each types (types are given in Table \ref{tableu3}), for the vectors $(a_1,a_2,a_3)$ of the same type can be dealt with in a similar manner.
For the rest of the proof, we let $K=\langle a_1,a_2,a_3\rangle$ and $J=\langle a_1,a_2,a_3,a_4\rangle$.

\noindent (${\bf 0'}$) Assume that $(a_1,a_2,a_3)=(1,1,9)$. If $a_i$ is divisible by 9 for some $i\ge 4$, then one may easily check that $\T(a_i)$ is redundant in $\E$ by using Lemma \ref{lemnew}.
Hence neither $a_4$ nor $a_5$ is divisible by 9.
Then one may show that
$$
5\in T^{\text{loc}}(1,1,9,a_4,a_5)\subseteq T^{\text{loc}}(\E),
$$
which is absurd since $\T(1,1,9,a_4,\dots,a_k)$ cannot represent 5.

\noindent ({\bf 1}) Assume that $(a_1,a_2,a_3)=(1,1,3)$.
Let $s$ be the positive integer such that $\ord_3(a_4)\in \{2s-2,2s-1\}$.
Take $n=\dfrac{69\cdot 3^{2s-2}-5}{8}$.
Since
$$
8n+5+a_4=69\cdot 3^{2s-2}+a_4\equiv 0\Mod{3^{2s-2}},
$$
one may easily check that $8n+5+a_4$ is represented by $J_3$.
Since $Q(J_q)=\z_q$ for any odd prime $q$ greater than 3, it follows that $8n+5+a_4$ is represented by $J_q$ for all odd primes $q$.
Thus we have
$$
n\in T^{\text{loc}}(1,1,3,a_4)\subset T^{\text{loc}}(\E).
$$
On the other hand, one may easily check that $8n+5=69\cdot 3^{2s-2}$ is not represented by $K_3$.
Hence $n$ cannot be represented by $\T(1,1,3)$ and thus
$$
a_4\le n<9\cdot 3^{2s-2}.
$$
Thus we have
$$
a_4\in \{3^{2s-2},2\cdot 3^{2s-2},\dots,8\cdot 3^{2s-2}\}.
$$
Note that one should exclude the cases of $a_4\in \{1,2\}$ since we are assuming that $3\le a_4$.
Now, $\T(1,1,3,a_4)$ is a primitive regular triangular form by Proposition \ref{propreg}.
Let $t$ be the positive integer such that $\ord_3(a_5)\in \{2t-2,2t-1\}$.
If $t<s$, then one may easily deduce that
$$
a_5\in \{3^{2t-2},2\cdot 3^{2t-2},\dots,8\cdot 3^{2t-2}\},
$$
and this implies that $\T(a_4)$ is redundant in $\T(1,1,3,a_4,a_5)$ and hence redundant in $\E$ also.
If $t\ge s$, then one may easily check that $a_5$ is redundant in $\T(1,1,3,a_4,a_5)$ and hence in $\E$.
This is absurd.

\noindent (${\bf 1'}$) Assume that $(a_1,a_2,a_3)=(1,1,12)$.
Let $s$ be the positive integer such that $\ord_3(a_4)\in \{2s-2,2s-1\}$.
If $s=1$, then one may easily check that
$$
5\in T^{\text{loc}}(1,1,12,a_4)\subset T^{\text{loc}}(\E),
$$
which is absurd since $\T(1,1,12,a_4,\dots,a_k)$ cannot represent 5.
From now on, we may assume that $s\ge 2$.
Take $n=\dfrac{6\cdot 3^{2s-2}-14}{8}\in \n$.
Then one may easily show that
$$
n\in T^{\text{loc}}(\E)-T(1,1,12).
$$
This implies that $a_4\le n$.
This is absurd since $n<3^{2s-2}$ whereas $\ord_3(a_4)\in \{2s-2,2s-1\}$.

\noindent (${\bf 2'}$) Assume that $(a_1,a_2,a_3)=(1,1,21)$.
Let $s$ be the positive integer such that $\ord_3(a_4)\in \{2s-2,2s-1\}$.
Take $n=\dfrac{87\cdot 3^{2s-2}-23}{8}$.
Then $8n+23$ is represented by $K_7$ since it is coprime to 7.
It follows that $8n+23+a_4$ is represented by $J_7$.
Since $8n+23+a_4\equiv 0\Mod{3^{2s-2}}$, $8n+23+a_4$ is represented by $J_3$.
Hence, $8n+23+a_4$ is represented by $J_q$ over $\z_q$ for every odd prime $q$.
Thus we have
$$
n\in T^{\text{loc}}(1,1,21,a_4)\subset T^{\text{loc}}(\E).
$$
However, $n$ is not represented by $\T(1,1,21)$ since $8n+23=87\cdot 3^{2s-2}$ is not represented by $K_3$.
Hence we have
$$
a_4\le n=\frac{87\cdot 3^{2s-2}-23}{8},
$$
and thus
$$
a_4\in \{3^{2s-2},2\cdot 3^{2s-2},\dots,10\cdot 3^{2s-2}\}-\{9\cdot 3^{2s-2}\}.
$$
From this follows that $Q(J_7)=\z_7$.
Since $8\cdot 5+23=63$ is represented by $K_3$, we have $8\cdot 5+23+a_4$ is represented by $J_3$.
Then $8\cdot 5+23+a_4$ is locally represented by $J$.
It forces that $5\in T^{\text{loc}}(\E)$, which is absurd.

\noindent (${\bf 3'}$) Assume that $(a_1,a_2,a_3)=(1,1,1)$.
Then clearly $a_4$ is redundant in $\E$ and this contradicts to the newness of $\E$.

\noindent (${\bf 4'}$) Assume that $(a_1,a_2,a_3)=(3,3,3)$.
We may assume that $a_4$ is coprime to 3 since otherwise $\T(a_4)$ is redundant in $\E$.
Take $n\in \{1,2\}$ with $n\equiv a_4\Mod 3$.
Then one may easily show that
$$
n\in T^{\text{loc}}(3,3,3,a_4)\subset T^{\text{loc}}(\E),
$$
which is absurd.

\noindent (${\bf 5'}$) Assume that $(a_1,a_2,a_3)=(3,6,15)$.
Since $\E$ is primitive, there is an index $i$ with $1\le i\le k$ such that $a_i\not\equiv 0\Mod 3$.
Take $n\in \{1,2\}$ such that $n\equiv a_i\Mod 3$.
Since at least one of $8n+24+a_4$ and $8(n+3)+24+a_4$ is coprime to 5,
at least one of them are locally represented by $\langle 3,6,15,a_i\rangle$.
Hence $n$ or $n+3$ is in $T^{\text{loc}}(3,6,15,a_i)$, which is a subset of $T(\E)$.
This is absurd since none of 1,2,4 and 5 is represented by $\E=\T(3,6,15,a_4,\dots,a_k)$.

\noindent (${\bf 6}$) Assume that $(a_1,a_2,a_3)=(1,6,18)$.
Note that $\psi(1,6,18)=43$ and one may easily check with the help of computer that
$$
1\le \psi(1,6,18,a_4)<a_4,\ \ a_4\in \{18,19,\dots,43\}-\{18,27,36\}.
$$
Hence we have $a_4\in \{18,27,36\}$, and thus $\T(1,6,18,a_4)$ is regular by Proposition \ref{propreg}.
Note that $a_5\not\equiv 0\Mod 9$ since otherwise $a_5$ is redundant to $\T(1,6,18,a_4)$.
If $a_5\equiv 2\Mod 3$, then one may easily show that $2\in T^{\text{loc}}(1,6,18,a_4,a_5)$, which is absurd.
If we take a integer $n$ with $0\le n\le 8$ such that
$$
\begin{cases}8n+25+a_4+a_5\equiv 6\Mod 9&\text{if}\ a_5\equiv 1\Mod 3,\\
8n+25+a_4\equiv 3\Mod 9&\text{if}\ a_5\equiv 3,6\Mod 9,\end{cases}
$$
then one may easily check that
$$
n\in T^{\text{loc}}(1,6,18,a_4,a_5)-T^{\text{loc}}(1,6,18,a_4).
$$
This is absurd.

\noindent (${\bf 6'}$) Assume that $(a_1,a_2,a_3)=(1,2,7)$.
Note that $\psi(1,2,7)=11$, and thus we have $a_4\in \{7,8,9,10,11\}$.
Hence $Q(J_7)=\z_7$.
One may easily check that
$$
4\in T^{\text{loc}}(1,2,7,a_4)\subset T^{\text{loc}}(\E),
$$
which is absurd.

Hence we have a contradiction in any cases and this completes the proof.
\end{proof}


\begin{proof}[Proof of Theorem \ref{thmold}]
If $\E$ is new, then $I=\{1,2,\dots,k\}$.
Assume that $\E$ is old.
Then $\E$ can be obtained from a new regular triangular form, say $\mathcal{F}$, by redundant insertions.
By Theorems \ref{thmrk3}-\ref{thmrk5}, $\mathcal{F}$ must appear in either Table \ref{tablerk3} or Table \ref{tablerk4}.
The divisibility of the coefficients $a_j$ for $j\in \{1,2,\dots,k\}-I$ immediately comes from Lemma \ref{lemnew}.
This completes the proof.
\end{proof}


\begin{table}[ht]
\caption{The 49 vectors $(a_1,a_2,a_3)$ for which $\T(a_1,a_2,a_3)$ is primitive regular}
\begin{tabular}{lllllll}
\hline
\multicolumn{7}{c}{$(a_1,a_2,a_3)$}\\
\hline
${}^{\dag}(1,1,1)$, & ${}^{\dag}(1,1,2)$, & ${}^{\dag}(1,1,3)$, & ${}^{\dag}(1,1,4)$, & ${}^{\dag}(1,1,5)$, & ${}^{\dag}(1,1,6)$, & $(1,1,9)$,\\
\hline
${}^{\dag}(1,1,12)$, & $(1,1,18)$, & ${}^{\dag}(1,1,21)$, & ${}^{\dag}(1,2,2)$, & ${}^{\dag}(1,2,3)$, & ${}^{\dag}(1,2,4)$, & ${}^{\dag}(1,2,5)$,\\
\hline
${}^{\dag}(1,2,10)$, & $(1,3,3)$, & ${}^{\dag}(1,3,4)$, & $(1,3,6)$, & $(1,3,9)$, & ${}^{\dag}(1,3,10)$, & $(1,3,12)$,\\
\hline
$(1,3,18)$, & $(1,3,27)$, & $(1,3,30)$, & ${}^{\dag}(1,4,6)$, & $(1,4,9)$, & $(1,5,5)$, & $(1,5,10)$,\\
\hline
$(1,5,25)$, & $(1,6,6)$, & $(1,6,9)$, & $(1,6,27)$, & $(1,9,9)$, & $(1,9,12)$, & $(1,9,18)$,\\
\hline
$(1,9,21)$, & $(1,21,21)$, & ${}^{\dag}(2,2,3)$, & $(2,3,3)$, & $(2,3,9)$, & $(2,3,12)$, & $(2,3,18)$,\\
\hline
 $(2,5,10)$, & $(2,5,15)$, & $(3,3,4)$, & $(3,3,7)$, & $(3,7,7)$, & $(3,7,63)$, & $(5,6,15)$\\
\hline
\end{tabular}
\label{tablerk3}
\end{table}

\clearpage


\end{document}